\documentclass[12pt]{amsart}%
\usepackage{amsfonts}
\usepackage{amsmath}
\usepackage{amssymb}
\usepackage{amscd}

\usepackage{graphicx}
\newtheorem{theorem}{Theorem}[section]
\newtheorem{lemma}[theorem]{Lemma}

\theoremstyle{Corollary}
\newtheorem{corollary}[theorem]{Corollary}
\newtheorem{proposition}[theorem]{Proposition}
\newtheorem{definition}[theorem]{Definition}

\newtheorem{remark}[theorem]{Remark}

\providecommand{\U}[1]{\protect\rule{.1in}{.1in}}
\theoremstyle{definition}
\theoremstyle{remark}
\numberwithin{equation}{section}

\begin{document}
\title[Legendrian mean curvature flow]{Legendrian mean curvature flow in $\eta$-Einstein Sasakian manifolds}
\author{Shu-Cheng Chang$^{1\ast}$}
\address{$^{1}$Department of Mathematics, National Taiwan University, Taipei 10617,
Taiwan and Mathematical Science Research Center, Chongqing University of
Technology, 400054, Chongqing, P.R. China}
\email{scchang@math.nthu.edu.tw }
\author{Yingbo Han$^{2}$}
\address{$^{2}$School of Mathematics and Statistics, Xinyang Normal University,
Xinyang,464000, Henan, P. R. China}
\email{yingbohan@163.com}
\author{Chin-Tung Wu$^{3\ast}$}
\address{$^{3}$Department of Applied Mathematics, National Pingtung University,
Pingtung 90003, Taiwan}
\email{ctwu@mail.nptu.edu.tw }
\thanks{$^{\ast}$Research supported in part by NSC. $^{2}$Yingbo Han is partially
supported by an NSFC 11971415 and Nanhu Scholars Program for Young Scholars of
{Xinyang Normal University}.}
\subjclass{Primary 53C44, Secondary 53C56.}
\keywords{Isotopic problem, $\eta$-Einstein Sasakian metric, Legendrian mean curavture
flow, minimal Legendrian submanifold.}
\maketitle

\begin{abstract}
Recently, there are a great deal of work done which connects the Legendrian
isotopic problem with contact invariants. The isotopic problem of Legendre
curve in a contact $3$-manifold was studies via the Legendrian curve
shortening flow which was introduced and studied by K. Smoczyk. On the other
hand, in
%TCIMACRO{\QTR{frametitle}{the SYZ Conjecture}}%
%BeginExpansion
the SYZ Conjecture%
%EndExpansion
,
%TCIMACRO{\QTR{frametitle}{one can model a special Lagrangian singularity
%locally as the }}%
%BeginExpansion
one can model a special Lagrangian singularity locally as the
%EndExpansion
special Lagrangian cones in $\mathbb{C}^{3}$. This can be characterized by its
link%
%TCIMACRO{\QTR{frametitle}{\ which is }}%
%BeginExpansion
\ which is
%EndExpansion
a minimal Legendrian surface in the $5$-sphere. Then in these points of view,
in this paper we will focus on the existence of the long-time solution and
asymptotic convergnce along the Legendrian mean curvature flow in higher
dimensional $\eta$-Einstein Sasakian $(2n+1)$-manifolds under the suitable
stability condition due to the Thomas-Yau conjecture.

\end{abstract}

\section{Introduction}

Legendrian submanifolds of contact manifolds and Lagrangian submanifolds of
symplectic manifolds are related by symplectization. More precisely, it is
proved in \cite{h} that for an $n$-dimensional submanifold $L$ of
$\mathbb{S}^{2n+1}$, the cone $C(L)\backslash\{0\}$ is Lagrangian in
$\mathbb{C}^{n+1}$ if and only if $L$ is Legendrian. Furthermore, there is a
$1$-$1$ correspondence between minimal Lagrangian cones in $\mathbb{C}^{n+1}$
and minimal Legendrian submanifolds in $\mathbb{S}^{2n+1}$ with the canonical
contact metric structure.

In
%TCIMACRO{\QTR{frametitle}{the SYZ Conjecture (\cite{syz}), in order to deal
%with }}%
%BeginExpansion
the SYZ Conjecture (\cite{syz}), in order to deal with
%EndExpansion
the difficulty which states that most of the special Lagrangian tori fibration
have singularities, one can model them locally as special Lagrangian cones
$C\subset\mathbb{C}^{3}$. Such a cone can be characterized by its link
\[
C\cap\mathbb{S}^{5}%
\]
which is a Legendrian surface. Then the link over a special Lagrangian
singularity is a minimal Legendrian surface in the sphere $\mathbb{S}^{5}$.

On the other hand, the standard Sasakian sphere $(\mathbb{S}^{2n+1},\alpha)$
is the compact contactization of the K\"{a}hler projective space
$(\mathbb{CP}^{n},\omega)$. Then the projection of the mean curvature vector
field of a Legendrian submanifold $L$ in the contactization of a K\"{a}hler
manifold $Z$ coincides with the mean curvature of the projected (immersed)
Lagrangian submanifold in $Z$. In particular, the projection of a minimal
Legendrian submanifold in the contactization is an immersed minimal Lagrangian
submanifold in $Z$.

We first mention the Legendrian isotopic problem which connects with a great
deal of work done in the area of contact invariants. It is known that any
closed $n$-manifold $L$ embeds in the standard contact space $\mathbb{R}%
^{2n+1}$ and are isotopic for any two such embeddings with $n\geq2.$ Moreover,
it is proved by Eliashberg-Fraser \cite{ef} that any Legendrian trivial knot
in a tight contact $3$-manifold with the same rotation number and
Thurston-Bennequin invariant are Legendrian isotopic. Furthermore, it is
observed by Colin-Giroux-Honda \cite{cgh} that for any contact $3$-manifold if
you fix a Thurston-Bennequin invariant, a rotation number, and a knot type,
there are only finitely many Legendrian knot types realizing this data. For
any $n>1,$ by the $h$-principle and constructing from immersion Lagrangian
submanifolds of $\mathbb{C}^{n},$ Ekholm-Etnyre-Sullivan \cite{ees} showed
that there is an infinite family of Legendrian embeddings of the $n$-sphere
($n$-torus) into $\mathbb{R}^{2n+1}$ that are not Legendrian isotopic even
though they have the same classical rotation number and Thurston-Bennequin invariant.

In view of previous observations, it is important but also very difficult to
construct a minimal Legendrian submanifold. In this paper, we will focus on
the existence problem of minimal Legendrian submanifolds of $\eta$-Einstein
Sasakian manifold through deformation of the mean curvature flow.

The mean curvature flow is perhaps the most important geometric evolution
equation of submanifolds in Riemannian and K\"{a}hler manifolds. Intuitively,
a family of smooth immersions $F_{t}:L\rightarrow(M,g)$ such that
\begin{equation}%
\begin{array}
[c]{c}%
\frac{d}{dt}F_{t}=H
\end{array}
\label{M}%
\end{equation}
where $H$ is the mean curvature vector of the immersion along $L_{t}=F_{t}(L)
$. It is the negative gradient flow of the volume functional of $L_{t}$ and
it's stationary solutions are minimal submanifolds. There are many results on
this equation which belongs to the most important equations in Geometric
Analysis. This evolution equation is one of the most important equations in
geometric analysis. There are many results on this equation. For the case of
hypersurfaces, we refer to \cite{hp} for a survey. Recently, the most
interesting class in higher codimension is the Lagrangian mean curvature flow
which preserves the class of Lagrangian submanifolds if the ambient manifold
is a K\"{a}hler-Einstein manifold (\cite{s1}). We refer to \cite{w} and
references therein.

In the case of Legendrian $n$-submanifolds of Sasakian $(2n+1)$-manifolds, in
contrast to the Lagrangian mean curvature flow, the mean curvature flow does
not preserve the Legendrian condition. Then it is the first paper by Smoczyk
\cite{s2} to consider a modified flow called the Legendrian mean curvature
flow which preserves the Legendrian condition if the ambient manifold is an
$\eta$-Einstein Sasakian manifold. The main idea is to minimize the volume
energy in the class of Legendrian immersions and then establish such a
Legendrian mean curvature flow for Legendre submanifolds. \

It was conjectured by Thomas-Yau \cite{ty} that for a compact Lagrangian
submanifold with the zero Maslov class in a Calabi-Yau K\"{a}hler manifold,
under some stability conditions, the Lagrangian mean curvature flow exists for
all time and converges to a special Lagrangian submanifold in its hamiltonian
deformation class. It is constructed by Neves \cite{ne1} that there is a
normalized monotone Lagrangian $L$ in $\mathbb{C}^{2}$ which is Hamiltonian
isotopic to a Clifford torus $\Sigma$ but the Lagrangian mean curvature flow
starting at $L$ develops a finite time singularity. However, Recently Evans
\cite{e} proved a Thomas-Yau type conjecture for\ monotone Lagrangian tori
satisfying a symmetry condition in the complex projective plane $\mathbb{CP}%
^{2}$. More precisely, it is proved that such tori exist for all time under
Lagrangian mean curvature flow with surgery, undergoing at most a finite
number of surgeries before flowing to a minimal Clifford torus in infinite time.

In this paper, we will focus questions on the long-time solution and
asymptotic convergence along the Legendrian mean curvature flow (\ref{27})
under suitable stability conditions. Firstly, it is important to know the
long-time behavior of the flow.

\begin{theorem}
\label{12} Let $F_{0}:L_{0}\rightarrow(M,\lambda,g,\mathbf{T},J)$ be a compact
Legendrian submanifold with the exact mean curvature form $H_{0}$ in an $\eta
$-Einstein Sasakian manifold $M$ which is either compact or bounded curvature.
Assume that the Legendrian mean curvature flow (\ref{27}) admits a smooth
solution on a maximal time interval $[0,T)$, $0<T\leq\infty$. Then the
following is true :\newline

\begin{enumerate}
\item Assume that there exists a positive constant $C_{0}$ such that
\begin{equation}%
\begin{array}
[c]{c}%
\max_{L_{t}}|A|^{2}\leq C_{0},\text{\ for all\ }t\in\lbrack0,T),
\end{array}
\label{31}%
\end{equation}
where $|A|^{2}$ is the squared norm of the second fundamental tensor $A$. Then
there exists a positive constant $C_{m}$ such that
\begin{equation}%
\begin{array}
[c]{c}%
\max_{L_{t}}|\nabla^{m}A|^{2}\leq C_{m},\text{\ for all\ }t\in\lbrack0,T).
\end{array}
\label{32}%
\end{equation}

\item If $T<\infty$, then
\begin{equation}%
\begin{array}
[c]{c}%
\lim_{t\rightarrow T}\max_{L_{t}}|A|^{2}=\infty.
\end{array}
\label{33}%
\end{equation}

\item If in addition to $(1)$ the initial mean curvature form of $L_{0}$ is
exact, and the induced Riemannian metrics $g_{ij}(x,t)$ on $L$ are all
uniformly equivalent, then $T=\infty$ and the Legendrian submanifolds $L_{t} $
converge smoothly and exponentially to a smooth compact minimal Legendrian
immersion $L_{\infty}\subset M$.
\end{enumerate}
\end{theorem}

Now we recall a Legendrian stability result from the paper of Ono \cite{ono1}.
Let $(M^{2n+1},\eta,\xi,\Phi,\overline{g})$ be an $\eta$-Einstein Sasakian
$(2n+1)$-manifold with the $\eta$-Einstein constant $K+2$ (Definition
\ref{d3}). If $K+2\leq0$, then any compact minimal Legendrian submanifold $L$
of $M$ is stable and strictly stable if $K+2<0$. Moreover, if $K+2>0,$ then
$L$ is Legendrian stable if and only if
\begin{equation}
\lambda_{1}(L)\geq K+2,\label{2023}%
\end{equation}
where $\lambda_{1}(L)$ is the first positive eigenvalue of the Hodge-Laplacian
$\Delta_{d}$ acting on $C^{\infty}(L)$ with respect to the induced metric in
$L_{0}.$

\begin{theorem}
\label{23} Let $(M^{2n+1},\eta,\xi,\Phi,\overline{g})$ be an $\eta$-Einstein
Sasakian $(2n+1)$-manifold with the $\eta$-Einstein constant $K+2<0$ and
$L_{0}$ be a compact Legendrian submanifold smoothly immersed in $M$. For any
positive constants $V_{0},$ $\Lambda_{0}>0$, there exists a small constant
$\epsilon_{0}=\epsilon_{0}(V_{0},\Lambda_{0},K_{5},K+2,\iota_{0})>0$ such that
if $L_{0}$ satisfies
\[%
\begin{array}
[c]{c}%
\text{\textrm{Vol}}(L_{0})\leq V_{0},\text{\ }|A|\leq\Lambda_{0},\text{\ }%
\int_{L_{0}}|H|^{2}d\mu\leq\epsilon_{0},
\end{array}
\]
where $A$ is the second fundamental form of $L$ in $M$, $\iota_{0}$ is the
lower bound of the injectivity radius of $(M,\overline{g})$ and $K_{5}=%
%TCIMACRO{\tsum \limits_{k=0}^{5}}%
%BeginExpansion
{\textstyle\sum\limits_{k=0}^{5}}
%EndExpansion
\sup|\bigtriangledown^{k}\overline{Rm}|<\infty.$ Then the Legendrian mean
curvature flow (\ref{27}) with the exact mean curvature form $H_{0}$ for $t=0$
will converge exponentially fast to a minimal Legendrian submanifold in $M$.
\end{theorem}

For the ambient $\eta$-Einstein Sasakian $(2n+1)$-manifold with the $\eta
$-Einstein constant $K+2\geq0,$ we have the following:

\begin{theorem}
\label{Thm 3.1} Let $(M^{2n+1},\eta,\xi,\Phi,\overline{g})$ be an $\eta
$-Einstein Sasakian $(2n+1)$-manifold with the $\eta$-Einstein constant
$K+2\geq0$ and $L_{0}$ be a compact Legendrian submanifold smoothly immersed
in $M$. For any positive constants $V_{0},$ $\Lambda_{0},$ $\delta_{0}>0$,
there exists a small constant $\epsilon_{0}=\epsilon_{0}(V_{0},\Lambda
_{0},\delta_{0},K_{5},K+2,\iota_{0})>0$ such that if $L_{0}$ satisfies
\begin{equation}%
\begin{array}
[c]{c}%
\lambda_{1}\geq K+2+\delta_{0},\text{ \textrm{Vol}}(L_{0})\leq V_{0}%
,\text{\ }|A|\leq\Lambda_{0},\text{\ }\int_{L_{0}}|H|^{2}d\mu\leq\epsilon_{0}.
\end{array}
\label{70}%
\end{equation}
Then the Legendrian mean curvature flow (\ref{27}) with the exact mean
curvature form $H_{0}$ for $t=0$ will converge exponentially fast to a minimal
Legendrian submanifold in $M$.
\end{theorem}

Let $(Z,\omega)$ be a K\"{a}hler-Einstein manifold of positive scalar
curvature. We consider its contactization $M$ as before. Suppose that
$\overline{L}$ is a orientable minimal Lagrangian submanifold in $Z$, then
$\overline{L}$ is Hamiltonian stable in $Z$ (\cite{oh1}) if and only if the
horizontal lift $L$ is Legendrian stable in $M$ (\cite{ono1}, \cite{k}). Then
there is a Sasaki analogue of the Lagrangian mean curvature flow (\cite{li})
that the assumption of Theorem \ref{Thm 3.1} on the first eigenvalue ensures
that the limit minimal Legendrian submanifold is strictly Legendrian stable.
Therefore, it is interesting to know whether we have the corresponding result
in this situation.

\begin{theorem}
\label{Thm 4.0 copy(1)} Let $(M^{2n+1},\eta,\xi,\Phi,\overline{g})$ be an
$\eta$-Einstein Sasakian $(2n+1)$-manifold with the $\eta$-Einstein constant
$K+2\geq0$. Suppose that $\overline{\varphi}:L_{0}\rightarrow M$ be a compact
minimal Legendrian submanifold with the first eigenvalue $\lambda_{1}=K+2$ and
$X$ is an essential Legendrian variation vector field of $\overline{\varphi
}(L_{0})$ defined as (\ref{d2}). Let $\varphi_{s}:L\rightarrow M,$
$s\in(-\delta,\delta),$ with $\varphi_{0}=\overline{\varphi}$ be a
one-parameter family of Legendrian deformations generated by $X$. Then there
exists $\epsilon_{0}=\epsilon_{0}(X,L_{0},M)>0 $ such that if $L_{s}%
=\varphi_{s}(L)\subset M$ satisfying%
\[%
\begin{array}
[c]{c}%
\left\Vert \varphi_{s}-\varphi_{0}\right\Vert _{C^{3}}\leq\epsilon_{0},
\end{array}
\]
then the Legendrian mean curvature flow with the initial Legendrian
submanifold $L_{s}$ will converge exponentially fast to a minimal Legendrian
submanifold in $M$.
\end{theorem}

Parts of the results here are served as a Legendrian analogue of the
Lagrangian mean curvature flow done by H. Li \cite{li} and the generalized
Lagrangian mean curvature flow done by T. Kajigaya and K. Kunikawa \cite{kk}.
The crucial step is to apply the smallness of the mean curvature vector in a
short time interval to get the exponential decay of the $L^{2}$-norm of the
mean curvature vector. Moreover, by combining the noncollapsing assumption
(\cite{p}) which can be removed in the proof of the main theorems, one can
derive all higher-order estimates of the second fundamental form from the
$L^{2}$-norm estimate of the mean curvature vector. Furthermore, one can show
that the exponential decay of the mean curvature vector implies that the
second fundamental form is uniformly bounded for any time interval and can
extend the solution for all time.

\section{Preliminaries}

We first recall notions as in \cite{s2} for some details. Let $(M^{2n+1}%
,\lambda)$ be a contact manifold endowed with a contact one-form $\lambda$
such that
\[
\lambda\wedge(d\lambda)^{n}\neq0
\]
defines a volume form on $M$ which defines a natural orientation. The contact
form $\lambda$ induces the $2n$-dimensional contact distribution or contact
subbundle $\xi$ over $M$, which is given by
\[
\xi_{p}=\ker\lambda_{p},
\]
where $\xi_{p}$ is the fiber of $\xi$ at each point $p\in M.$ Moreover, $\xi$
is non-integrable and $\omega:=d\lambda$ defines a symplectic vector bundle
$(\xi,\omega_{|\xi\oplus\xi})$. The Reeb vector field $\mathbf{T}_{\lambda}$
is defined by
\[
\lambda(\mathbf{T}_{\lambda})=1;d\lambda(\mathbf{T}_{\lambda},\cdot)=0
\]

\begin{definition}
A submanifold $F:L\rightarrow M$ of a contact manifold $(M^{2n+1},\lambda)$ is
called isotropic if it is tangent to $\xi,$ $i.e.$ $\lambda_{|TL}=0$ or
$F^{\ast}\lambda=0.$ In particular $F^{\ast}d\lambda=0$ as well. A Legendrian
submanifold $L$ is a maximally isotropic submanifold of dimension $n.$ Then
\[
TM^{2n+1}=TL^{n}\oplus NL^{n}=TL^{n}\oplus JTL^{n}\oplus\mathbb{R}%
\mathbf{T}_{\lambda}.
\]

\end{definition}

A Riemannian metric $g$ on $M$\ is called an adapted metric to a contact
manifold $(M^{2n+1},\lambda)$ if \
\[
g(\mathbf{T}_{\lambda},X)=\lambda(X)
\]
for all $X\in TM.$ That is to say that
\[
g^{\alpha\beta}\lambda_{\beta}=T_{\lambda}^{\alpha}:=\lambda^{\alpha}.
\]
Then the adapted metric
\[
g=\lambda\otimes\lambda+\omega(\cdot,J\cdot)
\]
and then
\[
g_{\alpha\beta}=\lambda_{\alpha}\lambda_{\beta}+\omega_{\alpha\gamma}J_{\beta
}^{\gamma}.
\]
Here $\widetilde{J}$ is an almost complex structure on the sympletic subbundle
$\xi$ which can be extend to a section $J\in\Gamma(T^{\ast}M\otimes TM)$ by
\[%
\begin{array}
[c]{c}%
J(X):=\widetilde{J}(\pi(X))
\end{array}
\]
for the projection $\pi:TM\rightarrow\xi$ with $\pi(X)=X-\lambda
(X)\mathbf{T}_{\lambda}.$ Then $J^{2}=-\pi$ and $J_{\alpha}^{\beta}J_{\gamma
}^{\alpha}=-\pi_{\gamma}^{\beta}.$

\begin{definition}
(\cite{s2},\cite{b}) $J$ is integrable if
\[
T(J)=0.
\]
Here $T(J):=N(J)+2\omega\otimes\mathbf{T}_{\lambda}$ and $N(J)$ is the
Nijenhuis tensor with $N(J)(X,Y):=J^{2}[X,Y]+[JX,JY]-J[X,JY]-J[JX,Y].$

A contact manifold $(M,\lambda,g,\mathbf{T},J)$ is called a Sasakian manifold
if $J$ is integrable.
\end{definition}

\begin{remark}
(\cite{bg}) Let $(M,g)$ be a Riemannian $(2n+1)$-manifold. There are also the
following equivalent statements for a Sasakian manifolds:

\begin{enumerate}
\item There exists a Killing vector field $\mathbf{T}_{\lambda}$ of unit
length on $M$ so that the tensor field of type $(1,1)$, defined by
\[
J(Y)=\nabla_{Y}\mathbf{T}_{\lambda}%
\]
satisfies the condition%
\[
(\nabla_{X}J)(Y)=g(\mathbf{T}_{\lambda},Y)X-g(X,Y)\mathbf{T}_{\lambda}%
\]
for any pair of vector fields $X$ and $Y$ on $M$.

\item There exists a Killing vector field $\mathbf{T}_{\lambda}$ of unit
length on $M$\ so that the Riemann curvature satisfies the condition%
\[
R(X,\mathbf{T}_{\lambda})Y=g(\mathbf{T}_{\lambda},Y)X-g(X,Y)\mathbf{T}%
_{\lambda}%
\]
for any pair of vector fields $X$ and $Y$ on $M$.

\item The cone%
\[
(C(M),\overline{g}):=(\mathbb{R}^{+}\times M\mathbf{,}dr^{2}+r^{2}g)
\]
such that $(C(M),\overline{g},\overline{J})$ is K\"{a}hler. Note that
$\{r=1\}=\{1\} \times M\subset C(M)$. Define the Reeb vector field
\[%
\begin{array}
[c]{c}%
\mathbf{T}_{\lambda}=\overline{J}(\frac{\partial}{\partial r})
\end{array}
\]
and the contact $1$-form
\[
\lambda(Y)=g(\mathbf{T}_{\lambda},Y).
\]
Then $\lambda(\mathbf{T}_{\lambda})=1$\ and\ $d\lambda(\mathbf{T}_{\lambda
},\cdot)=0.$ $\mathbf{T}_{\lambda}$ is killing vector field with unit length.
\end{enumerate}
\end{remark}

We first recall the following results for Sasakian $(2n+1)$-manifold.

\begin{lemma}
(\cite{s2}) Let $(M,\lambda,g,\mathbf{T},J)$ be a Sasakian $(2n+1)$-manifold
and $g=\lambda\otimes\lambda+\omega(\cdot,J\cdot)$ the corresponding adapted
metric with $\omega=d\lambda$. Then the following relations hold%
\begin{equation}%
\begin{array}
[c]{c}%
R_{\gamma\alpha\beta}^{\epsilon}\lambda_{\epsilon}=g_{\gamma\beta}%
\lambda_{\alpha}-g_{\gamma\alpha}\lambda_{\beta},
\end{array}
\label{L0}%
\end{equation}%
\begin{equation}%
\begin{array}
[c]{c}%
R_{\gamma\alpha\beta}^{\epsilon}\omega_{\epsilon\delta}+R_{\delta\alpha\beta
}^{\epsilon}\omega_{\gamma\epsilon}=-g_{\beta\delta}\omega_{\alpha\gamma
}+g_{\beta\gamma}\omega_{\alpha\delta}+g_{\alpha\delta}\omega_{\beta\gamma
}-g_{\alpha\gamma}\omega_{\beta\delta},
\end{array}
\label{L1}%
\end{equation}%
\begin{equation}%
\begin{array}
[c]{c}%
J_{\epsilon}^{\beta}R_{\gamma\alpha\beta}^{\epsilon}=R_{\alpha}^{\epsilon
}\omega_{\epsilon\gamma}-(2n-1)\omega_{\alpha\gamma},
\end{array}
\label{L2}%
\end{equation}%
\begin{equation}%
\begin{array}
[c]{c}%
J_{\epsilon}^{\beta}R_{\beta\alpha\gamma}^{\epsilon}=R_{\gamma}^{\epsilon
}\omega_{\alpha\epsilon}-R_{\alpha}^{\epsilon}\omega_{\gamma\epsilon
}-2(2n-1)\omega_{\alpha\gamma}.
\end{array}
\label{L3}%
\end{equation}

\end{lemma}

\begin{proof}
One observes that $\nabla_{\alpha}\lambda_{\beta}=\omega_{\alpha\beta} $ and
$d\omega=0$ imply%
\[%
\begin{array}
[c]{lll}%
g_{\gamma\alpha}\lambda_{\beta}-g_{\gamma\beta}\lambda_{\alpha} & = &
\nabla_{\gamma}\omega_{\beta\alpha}\\
& = & \nabla_{\alpha}\omega_{\beta\gamma}-\nabla_{\beta}\omega_{\alpha\gamma
}\\
& = & \nabla_{\alpha}\nabla_{\beta}\lambda_{\gamma}-\nabla_{\beta}%
\nabla_{\alpha}\lambda_{\gamma}\\
& = & -R_{\gamma\alpha\beta}^{\epsilon}\lambda_{\epsilon}.
\end{array}
\]
With the same compute
\[%
\begin{array}
[c]{lll}%
-R_{\gamma\alpha\beta}^{\epsilon}\omega_{\epsilon\delta}-R_{\delta\alpha\beta
}^{\epsilon}\omega_{\gamma\epsilon} & = & \nabla_{\alpha}\nabla_{\beta}%
\omega_{\gamma\delta}-\nabla_{\beta}\nabla_{\alpha}\omega_{\gamma\delta}\\
& = & \nabla_{\alpha}(g_{\beta\delta}\lambda_{\gamma}-g_{\beta\gamma}%
\lambda_{\delta})-\nabla_{\beta}(g_{\alpha\delta}\lambda_{\gamma}%
-g_{\alpha\gamma}\lambda_{\delta})\\
& = & g_{\beta\delta}\nabla_{\alpha}\lambda_{\gamma}-g_{\beta\gamma}%
\nabla_{\alpha}\lambda_{\delta}-g_{\alpha\delta}\nabla_{\beta}\lambda_{\gamma
}+g_{\alpha\gamma}\nabla_{\beta}\lambda_{\delta},
\end{array}
\]
and (\ref{L1}) follows from the relation $\nabla_{\alpha}\lambda_{\gamma
}=\omega_{\alpha\gamma}.$ By taking the trace of (\ref{L1}) with respective to
$\beta$ and $\delta,$ one obtains (\ref{L2}). To prove (\ref{L3}) we apply the
Bianchi identity to get%
\[%
\begin{array}
[c]{c}%
J_{\epsilon}^{\beta}R_{\beta\alpha\gamma}^{\epsilon}=J_{\epsilon}^{\beta
}(R_{\gamma\alpha\beta}^{\epsilon}-R_{\alpha\gamma\beta}^{\epsilon})
\end{array}
\]
and then (\ref{L3}) is a consequence of (\ref{L2}).
\end{proof}

The second fundamental tensor $A=\nabla dF$ in local coordinates is
\[%
\begin{array}
[c]{c}%
A=A_{ij}^{\alpha}dx^{i}\otimes dx^{j}\otimes\frac{\partial}{\partial
y^{\alpha}}\text{ with }A_{ij}^{\alpha}=\nabla_{i}F_{j}^{\alpha}.
\end{array}
\]
Moreover, since $\lambda_{\alpha}A_{ij}^{\alpha}=g_{\alpha\beta}F_{i}^{\alpha
}A_{jk}^{\beta}=0,$ we can write the second fundamental form as%
\[%
\begin{array}
[c]{c}%
h=h_{ijk}dx^{i}\otimes dx^{j}\otimes dx^{k}\text{ with }h_{ijk}=\left\langle
A_{ij}^{\alpha}\frac{\partial}{\partial y^{\alpha}},JF_{k}\right\rangle
=-\omega_{\alpha\beta}F_{k}^{\alpha}A_{ij}^{\beta}.
\end{array}
\]
We recall the Gauss equation and the Codazzi equation are%
\begin{equation}%
\begin{array}
[c]{c}%
R_{ijkl}=R_{\alpha\beta\gamma\delta}F_{i}^{\alpha}F_{j}^{\beta}F_{k}^{\gamma
}F_{l}^{\delta}+(h_{ik}^{m}h_{mjl}-h_{il}^{m}h_{mjk}),
\end{array}
\label{G}%
\end{equation}%
\begin{equation}%
\begin{array}
[c]{c}%
\nabla_{l}h_{ijk}-\nabla_{j}h_{ilk}=-\omega_{\alpha\beta}R_{\gamma
\delta\epsilon}^{\beta}F_{i}^{\alpha}F_{k}^{\gamma}F_{l}^{\delta}%
F_{j}^{\epsilon}=-R_{\beta\gamma\delta\epsilon}v_{i}^{\beta}F_{k}^{\gamma
}F_{l}^{\delta}F_{j}^{\epsilon}.
\end{array}
\label{C}%
\end{equation}

\begin{lemma}
(\cite{s2}) Let $(M,\lambda,g,\mathbf{T},J)$ be a Sasakian $(2n+1)$-manifold
and $L$ be a smooth $n$-dimensional manifold. Suppose that $(f,\theta)$ is a
smooth family of pairs consisting of functions $f$ and 1-forms $\theta$ on
$L$. Assume that $F_{t}:L\rightarrow M$, $t\in\lbrack0,\varepsilon)$ is a
smooth family of immersions into $M$ such that
\begin{equation}%
\begin{array}
[c]{c}%
\frac{d}{dt}F_{t}(x)=\theta^{k}v_{k}+f\mathbf{T}.
\end{array}
\label{1}%
\end{equation}
If $L_{0}=F_{0}(L)$ is Legendre. Then $L_{t}=F_{t}(L)$ is Legendre for all
$t\in\lbrack0,\varepsilon)$ if and only if $df=2\theta$.
\end{lemma}

The following evolution equations of Legendre immersions into a Sasakian
manifold evolves according to (\ref{1}) in \cite{s2}.

\begin{lemma}
\label{4} If a smooth family of Legendre immersions into a Sasakian manifold
$(M,\lambda,g,\mathbf{T},J)$ evolves according to (\ref{1}), then%
\begin{equation}%
\begin{array}
[c]{lll}%
\frac{\partial}{\partial t}g_{ij} & = & 2\nabla^{k}fh_{kij},
\end{array}
\label{2}%
\end{equation}%
\begin{equation}%
\begin{array}
[c]{lll}%
\frac{\partial}{\partial t}h_{ijk} & = & -\nabla_{j}\nabla_{k}\nabla
_{i}f+\nabla^{l}f(h_{lim}h_{jk}^{m}+h_{lkm}h_{ij}^{m})\\
&  & -2g_{ik}\nabla_{j}f-g_{ij}\nabla_{k}f-\nabla^{l}fR_{\alpha\beta
\gamma\delta}v_{i}^{\alpha}F_{k}^{\beta}v_{l}^{\gamma}F_{j}^{\delta},
\end{array}
\label{3}%
\end{equation}%
\begin{equation}%
\begin{array}
[c]{lll}%
\frac{\partial}{\partial t}H_{j} & = & -\nabla_{j}\Delta f-2\nabla_{j}%
f-\frac{1}{2}R_{\alpha\beta}(F_{l}^{\alpha}F_{j}^{\beta}+v_{l}^{\alpha}%
v_{j}^{\beta})\nabla^{l}f\\
& = & -\Delta\nabla_{j}f-2\nabla_{j}f+\frac{1}{2}R_{\alpha\beta}(F_{l}%
^{\alpha}F_{j}^{\beta}-v_{l}^{\alpha}v_{j}^{\beta})\nabla^{l}f\\
&  & +(H_{m}h_{jl}^{m}-h_{jm}^{k}h_{lk}^{m})\nabla^{l}f.
\end{array}
\label{5}%
\end{equation}

\end{lemma}

\begin{proof}
For the mean curvature form $H_{j},$ use (\ref{2}) and (\ref{3}), we compute%
\[%
\begin{array}
[c]{lll}%
\frac{\partial}{\partial t}H_{j} & = & \frac{\partial}{\partial t}%
(g^{ik}h_{ijk})\\
& = & h_{ijk}\frac{\partial}{\partial t}g^{ik}+g^{ik}\frac{\partial}{\partial
t}h_{ijk}\\
& = & -2\nabla^{l}fh_{lk}^{m}h_{jm}^{k}-\nabla_{j}\Delta f+2\nabla^{l}%
fh_{lk}^{m}h_{jm}^{k}\\
&  & -(2n+1)\nabla_{j}f-\nabla^{l}fv_{l}^{\gamma}F_{j}^{\delta}g^{ik}%
R_{\gamma\delta\beta\epsilon}v_{i}^{\beta}F_{k}^{\epsilon}\\
& = & -\nabla_{j}\Delta f-(2n+1)\nabla_{j}f\\
&  & +\frac{1}{2}\nabla^{l}fv_{l}^{\gamma}F_{j}^{\delta}(R_{\delta}^{\epsilon
}\omega_{\gamma\epsilon}-R_{\gamma}^{\epsilon}\omega_{\delta\epsilon
}+2(2n-1)\omega_{\delta\gamma})\\
& = & -\nabla_{j}\Delta f-2\nabla_{j}f-\frac{1}{2}\nabla^{l}fR_{\alpha\beta
}(F_{l}^{\alpha}F_{j}^{\beta}+v_{l}^{\alpha}v_{j}^{\beta}).
\end{array}
\]
where we use the equation%
\begin{equation}%
\begin{array}
[c]{c}%
g^{ik}R_{\gamma\delta\beta\epsilon}v_{i}^{\beta}F_{k}^{\epsilon}=-\frac{1}%
{2}J_{\beta}^{\sigma}R_{\sigma\gamma\delta}^{\beta}=\frac{1}{2}(R_{\gamma
}^{\epsilon}\omega_{\delta\epsilon}-R_{\delta}^{\epsilon}\omega_{\gamma
\epsilon}-2(2n-1)\omega_{\delta\gamma}).
\end{array}
\label{L4}%
\end{equation}
To prove the second equation of (\ref{5}). By taking the trace of the Gauss
equation (\ref{G})
\begin{equation}
R_{ik}=g^{jl}R_{ijkl}=R_{\alpha\beta}F_{i}^{\alpha}F_{k}^{\beta}+H_{m}%
h_{ik}^{m}-h_{il}^{m}h_{km}^{l}.\label{h}%
\end{equation}
Then it follows from the relation $\nabla_{j}\Delta f=\Delta\nabla_{j}%
f-R_{jl}\nabla^{l}f$.
\end{proof}

\begin{corollary}
If $F_{t}:L\rightarrow(M,\lambda,g,\mathbf{T},J)$, $t\in\lbrack0,\varepsilon)$
be a Legendrian immersions into an $\eta$-Einstein Sasakian manifold, then the
mean curvature form $H$ is closed.
\end{corollary}

\begin{proof}
It only to show that $\nabla_{l}H_{j}-\nabla_{j}H_{l}=0.$ By applying the
Codazzi equation (\ref{C}) and (\ref{L4}), one get%
\[%
\begin{array}
[c]{lll}%
\nabla_{l}H_{j}-\nabla_{j}H_{l} & = & g^{ik}(\nabla_{l}h_{ijk}-\nabla
_{j}h_{ilk})\\
& = & -g^{ik}R_{\beta\gamma\delta\epsilon}v_{i}^{\beta}F_{k}^{\gamma}%
F_{l}^{\delta}F_{j}^{\epsilon}\\
& = & -g^{ik}R_{\gamma\delta\beta\epsilon}v_{i}^{\beta}F_{k}^{\epsilon}%
F_{l}^{\gamma}F_{j}^{\delta}\\
& = & \frac{1}{2}(R_{\delta}^{\epsilon}\omega_{\gamma\epsilon}-R_{\gamma
}^{\epsilon}\omega_{\delta\epsilon}+2(2n-1)\omega_{\delta\gamma})F_{l}%
^{\gamma}F_{j}^{\delta}\\
& = & \frac{1}{2}R_{\alpha\beta}(v_{l}^{\alpha}F_{j}^{\beta}-F_{l}^{\alpha
}v_{j}^{\beta})
\end{array}
\]
and if $M$ is an $\eta$-Einstein Sasakian manifold, then%
\[%
\begin{array}
[c]{c}%
\nabla_{l}H_{j}-\nabla_{j}H_{l}=\frac{1}{2}Kg_{\alpha\beta}(v_{l}^{\alpha
}F_{j}^{\beta}-F_{l}^{\alpha}v_{j}^{\beta})=0.
\end{array}
\]

\end{proof}

By applying the Codazzi equation (\ref{C}) and the rule for interchanging
derivatives we can get the Simons type identity.

\begin{theorem}
For a Legendrian immersion $L$ into a Sasakian manifold $(M,\lambda
,g,\mathbf{T},J)$, we have
\begin{equation}%
\begin{array}
[c]{ll}
& \nabla_{i}\nabla_{j}H_{k}\\
= & \Delta h_{ijk}+(h_{mk}^{s}h_{is}^{l}-h_{mi}^{s}h_{ks}^{l})h_{jl}^{m}\\
& +(h_{mj}^{s}h_{is}^{l}-h_{mi}^{s}h_{js}^{l})h_{kl}^{m}+(h_{is}^{m}h_{ml}%
^{s}-H_{m}h_{il}^{m})h_{jk}^{l}\\
& -R_{\alpha\beta}[F_{i}^{\alpha}F_{l}^{\beta}h_{jk}^{l}-\frac{1}{2}%
(F_{k}^{\alpha}F_{s}^{\beta}h_{ij}^{s}-v_{s}^{\alpha}v_{j}^{\beta}h_{ik}%
^{s})]\\
& -R_{\alpha\beta\gamma\delta}[F_{l}^{\alpha}F_{i}^{\gamma}F_{m}^{\delta
}(F_{k}^{\beta}h_{j}^{ml}+F_{j}^{\beta}h_{k}^{ml})+F_{k}^{\beta}F_{j}^{\gamma
}(F_{s}^{\alpha}F_{l}^{\delta}-v_{l}^{\alpha}v_{s}^{\delta})h_{i}^{sl}]\\
& -R_{\alpha\beta\gamma\delta}[F_{i}^{\gamma}F_{l}^{\delta}(F_{s}^{\alpha
}F_{k}^{\beta}h_{j}^{sl}-v_{j}^{\alpha}v_{s}^{\beta}h_{k}^{sl})-v_{j}^{\alpha
}F_{k}^{\beta}(v_{s}^{\gamma}F_{l}^{\delta}h_{i}^{sl}+F_{i}^{\gamma}%
v_{s}^{\delta}H^{s})]\\
& -\frac{1}{2}J_{\epsilon}^{\beta}\nabla_{\sigma}R_{\delta\gamma\beta
}^{\epsilon}F_{i}^{\sigma}F_{j}^{\delta}F_{k}^{\gamma}+\frac{1}{2}%
\nabla^{\epsilon}R_{\epsilon\delta\beta\gamma}F_{i}^{\delta}v_{j}^{\beta}%
F_{k}^{\gamma}.
\end{array}
\label{b}%
\end{equation}

\end{theorem}

\begin{proof}
By using the Codazzi equation (\ref{C}) and the rule for interchanging
derivatives, we then obtain%
\[%
\begin{array}
[c]{lll}%
\nabla_{i}\nabla_{j}H_{k} & = & \nabla_{i}\nabla_{j}h_{mk}^{m}\\
& = & \nabla_{i}\nabla_{m}h_{jk}^{m}-g^{ml}\nabla_{i}(R_{\beta\gamma
\delta\epsilon}v_{l}^{\beta}F_{k}^{\gamma}F_{j}^{\delta}F_{m}^{\epsilon})\\
& = & \nabla_{m}\nabla_{i}h_{jk}^{m}-R_{kim}^{l}h_{jl}^{m}-R_{jim}^{l}%
h_{lk}^{m}-R_{il}h_{jk}^{l}\\
&  & -g^{ml}\nabla_{i}(R_{\beta\gamma\delta\epsilon}v_{l}^{\beta}F_{k}%
^{\gamma}F_{j}^{\delta}F_{m}^{\epsilon})\\
& = & \Delta h_{ijk}-R_{kim}^{l}h_{jl}^{m}-R_{jim}^{l}h_{lk}^{m}-R_{il}%
h_{jk}^{l}\\
&  & -g^{ml}\nabla_{i}(R_{\beta\gamma\delta\epsilon}v_{l}^{\beta}F_{k}%
^{\gamma}F_{j}^{\delta}F_{m}^{\epsilon})\\
&  & -g^{ml}\nabla_{m}(R_{\beta\gamma\delta\epsilon}v_{j}^{\beta}F_{k}%
^{\gamma}F_{i}^{\delta}F_{l}^{\epsilon}).
\end{array}
\]
Gauss equation (\ref{G}) gives
\[%
\begin{array}
[c]{c}%
R_{kim}^{l}h_{jl}^{m}=R_{\alpha\beta\gamma\delta}F_{l}^{\alpha}F_{k}^{\beta
}F_{i}^{\gamma}F_{m}^{\delta}h_{j}^{ml}+(h_{mi}^{s}h_{ks}^{l}-h_{mk}^{s}%
h_{is}^{l})h_{jl}^{m}.
\end{array}
\]
By the rule of derivatives $\nabla_{i}F_{k}^{\gamma}=A_{ik}^{\gamma}%
=-h_{ik}^{s}v_{s}^{\gamma}$, $\nabla_{i}v_{l}^{\beta}=-g_{il}\lambda^{\beta
}+h_{il}^{s}F_{s}^{\beta}$ and (\ref{L0})%
\[%
\begin{array}
[c]{lll}%
R_{\beta\gamma\delta\epsilon}\lambda^{\beta}F_{k}^{\gamma}F_{j}^{\delta}%
F_{m}^{\epsilon} & = & (g_{\gamma\epsilon}\lambda_{\delta}-g_{\gamma\delta
}\lambda_{\epsilon})F_{k}^{\gamma}F_{j}^{\delta}F_{m}^{\epsilon}\\
& = & g_{km}\lambda_{j}-g_{kj}\lambda_{m}=0.
\end{array}
\]
Also from (\ref{L2})%
\[%
\begin{array}
[c]{c}%
g^{ml}R_{\beta\gamma\delta\epsilon}v_{l}^{\beta}F_{m}^{\epsilon}=\frac{1}%
{2}J_{\sigma}^{\beta}R_{\delta\gamma\beta}^{\sigma}=\frac{1}{2}(R_{\gamma
}^{\epsilon}\omega_{\epsilon\delta}+(2n-1)\omega_{\delta\gamma})
\end{array}
\]
and%
\[%
\begin{array}
[c]{c}%
g^{ml}\nabla_{m}R_{\beta\gamma\delta\epsilon}F_{l}^{\epsilon}=g^{ml}%
\nabla_{\sigma}R_{\beta\gamma\delta\epsilon}F_{m}^{\sigma}F_{l}^{\epsilon
}=\frac{1}{2}\nabla_{\sigma}R_{\beta\gamma\delta\epsilon}g^{\epsilon\sigma
}=-\frac{1}{2}\nabla^{\epsilon}R_{\epsilon\delta\beta\gamma}.
\end{array}
\]
All these will imply
\[%
\begin{array}
[c]{lll}%
\nabla_{i}\nabla_{j}H_{k} & = & \Delta h_{ijk}-R_{\alpha\beta\gamma\delta
}F_{l}^{\alpha}F_{i}^{\gamma}F_{m}^{\delta}(F_{k}^{\beta}h_{j}^{ml}%
+F_{j}^{\beta}h_{k}^{ml})+(h_{mk}^{s}h_{is}^{l}-h_{mi}^{s}h_{ks}^{l}%
)h_{jl}^{m}\\
&  & +(h_{mj}^{s}h_{is}^{l}-h_{mi}^{s}h_{js}^{l})h_{kl}^{m}-(R_{\alpha\beta
}F_{i}^{\alpha}F_{l}^{\beta}+H_{m}h_{il}^{m}-h_{is}^{m}h_{ml}^{s})h_{jk}^{l}\\
&  & -\frac{1}{2}J_{\epsilon}^{\beta}\nabla_{\sigma}R_{\delta\gamma\beta
}^{\epsilon}F_{i}^{\sigma}F_{j}^{\delta}F_{k}^{\gamma}-R_{\beta\gamma
\delta\epsilon}F_{k}^{\gamma}F_{j}^{\delta}(F_{s}^{\beta}F_{l}^{\epsilon
}-v_{l}^{\beta}v_{s}^{\epsilon})h_{i}^{sl}\\
&  & +\frac{1}{2}R_{\alpha\beta}(F_{k}^{\alpha}F_{s}^{\beta}h_{ij}^{s}%
-v_{s}^{\alpha}v_{j}^{\beta}h_{ik}^{s})+\frac{1}{2}\nabla^{\epsilon
}R_{\epsilon\delta\beta\gamma}F_{i}^{\delta}v_{j}^{\beta}F_{k}^{\gamma}\\
&  & -R_{\beta\gamma\delta\epsilon}[F_{i}^{\delta}F_{l}^{\epsilon}%
(F_{s}^{\beta}F_{k}^{\gamma}h_{j}^{sl}-v_{j}^{\beta}v_{s}^{\gamma}h_{k}%
^{sl})-v_{j}^{\beta}F_{k}^{\gamma}(v_{s}^{\delta}F_{l}^{\epsilon}h_{i}%
^{sl}+F_{i}^{\delta}v_{s}^{\epsilon}H^{s})],
\end{array}
\]
which yields the Simons type identity (\ref{b}).
\end{proof}

\begin{proposition}
If a smooth family of Legendrian immersions into a Sasakian manifold
$(M,\lambda,g,\mathbf{T},J)$ evolves according to (\ref{1}) with
$H_{j}=-\nabla_{j}f$, then

\begin{enumerate}
\item
\begin{equation}%
\begin{array}
[c]{c}%
\frac{\partial}{\partial t}g_{ij}=-2H^{k}h_{kij}.
\end{array}
\label{i}%
\end{equation}

\item
\begin{equation}%
\begin{array}
[c]{c}%
\frac{\partial}{\partial t}H_{j}=\Delta H_{j}+2H_{j}+\frac{1}{2}R_{\alpha
\beta}(v_{l}^{\alpha}v_{j}^{\beta}-F_{l}^{\alpha}F_{j}^{\beta})H^{l}%
+(h_{lm}^{k}h_{jk}^{m}-H_{m}h_{jl}^{m})H^{l}.
\end{array}
\label{j}%
\end{equation}

\item
\begin{equation}%
\begin{array}
[c]{ll}
& \frac{\partial}{\partial t}h_{ijk}\\
= & \Delta h_{ijk}-(h_{lmk}h_{ij}^{m}+h_{lmj}h_{ik}^{m}+h_{lmi}h_{jk}%
^{m})H^{l}+2g_{jk}H_{i}+g_{ik}H_{j}\\
& +(h_{mk}^{s}h_{is}^{l}-h_{mi}^{s}h_{ks}^{l})h_{jl}^{m}+(h_{mj}^{s}h_{is}%
^{l}-h_{mi}^{s}h_{js}^{l})h_{kl}^{m}+h_{is}^{m}h_{ml}^{s}h_{jk}^{l}\\
& -R_{\alpha\beta}[F_{i}^{\alpha}F_{l}^{\beta}h_{jk}^{l}-\frac{1}{2}%
(F_{k}^{\alpha}F_{s}^{\beta}h_{ij}^{s}-v_{s}^{\alpha}v_{j}^{\beta}h_{ik}%
^{s})]+R_{\alpha\beta\gamma\delta}v_{k}^{\alpha}F_{j}^{\beta}v_{l}^{\gamma
}F_{i}^{\delta}H^{l}\\
& -R_{\alpha\beta\gamma\delta}[F_{l}^{\alpha}F_{i}^{\gamma}F_{m}^{\delta
}(F_{k}^{\beta}h_{j}^{ml}+F_{j}^{\beta}h_{k}^{ml})+F_{k}^{\beta}F_{j}^{\gamma
}(F_{s}^{\alpha}F_{l}^{\delta}-v_{l}^{\alpha}v_{s}^{\delta})h_{i}^{sl}]\\
& -R_{\alpha\beta\gamma\delta}[F_{i}^{\gamma}F_{l}^{\delta}(F_{s}^{\alpha
}F_{k}^{\beta}h_{j}^{sl}-v_{j}^{\alpha}v_{s}^{\beta}h_{k}^{sl})-v_{j}^{\alpha
}F_{k}^{\beta}(v_{s}^{\gamma}F_{l}^{\delta}h_{i}^{sl}+F_{i}^{\gamma}%
v_{s}^{\delta}H^{s})]\\
& -\frac{1}{2}J_{\epsilon}^{\beta}\nabla_{\sigma}R_{\delta\gamma\beta
}^{\epsilon}F_{i}^{\sigma}F_{j}^{\delta}F_{k}^{\gamma}+\frac{1}{2}%
\nabla^{\epsilon}R_{\epsilon\delta\beta\gamma}F_{i}^{\delta}v_{j}^{\beta}%
F_{k}^{\gamma}.
\end{array}
\label{k}%
\end{equation}

\end{enumerate}
\end{proposition}

\begin{corollary}
If $F_{t}:L\rightarrow(M,\lambda,g,\mathbf{T},J)$, $t\in\lbrack0,\varepsilon)$
is a smooth family of Legendrian immersions into a Sasakian $\eta$-Einstein
manifold with $\eta$-Einstein constant $K+2$ that evolves according to
(\ref{1}) with $H=-df$, then

\begin{enumerate}
\item The mean curvature form $H$ satisfies
\begin{equation}%
\begin{array}
[c]{c}%
\frac{\partial}{\partial t}\left\vert H\right\vert ^{2}=\Delta\left\vert
H\right\vert ^{2}-2\left\vert \nabla H\right\vert ^{2}+4\left\vert
H\right\vert ^{2}+2H^{i}H^{j}h_{im}^{l}h_{jl}^{m}.
\end{array}
\label{H}%
\end{equation}

\item The second fundamental tensor $A$ satisfies%
\begin{equation}%
\begin{array}
[c]{lll}%
\frac{\partial}{\partial t}\left\vert A\right\vert ^{2} & = & \Delta\left\vert
A\right\vert ^{2}-2\left\vert \nabla A\right\vert ^{2}-2(K+2)\left\vert
A\right\vert ^{2}+2h_{is}^{m}h_{ml}^{s}h_{jk}^{l}h^{ijk}\\
&  & +6\left\vert H\right\vert ^{2}-2R_{\alpha\beta\gamma\delta}%
(4F_{m}^{\alpha}F_{j}^{\beta}+v_{m}^{\alpha}v_{j}^{\beta})F_{k}^{\gamma}%
F_{l}^{\delta}h_{i}^{ml}h^{ijk}\\
&  & +(\nabla^{\epsilon}R_{\epsilon\delta\beta\gamma}F_{i}^{\delta}%
v_{j}^{\beta}-J_{\epsilon}^{\beta}\nabla_{\sigma}R_{\delta\gamma\beta
}^{\epsilon}F_{i}^{\sigma}F_{j}^{\delta})F_{k}^{\gamma}h^{ijk}.
\end{array}
\label{A}%
\end{equation}

\item The volume form $d\mu$ satisfies the evolution equation
\begin{equation}%
\begin{array}
[c]{c}%
\frac{\partial}{\partial t}d\mu=-\left\vert H\right\vert ^{2}d\mu.
\end{array}
\label{V}%
\end{equation}

\end{enumerate}
\end{corollary}

Let $F_{0}:L_{0}\rightarrow(M,\lambda,g,\mathbf{T},J)$ be a Legendrian
immersion with the mean curvature form $H_{0}$. We consider a smooth
deformation of Legendrian immersions into a Sasakian manifold
\[%
\begin{array}
[c]{c}%
F_{t}:L^{n}\rightarrow(M,\lambda,g,\mathbf{T},J)
\end{array}
\]
which satisfies the following equations.
\begin{equation}%
\begin{array}
[c]{c}%
\frac{d}{dt}F_{t}(x)=\nabla^{k}fv_{k}+2f\mathbf{T}%
\end{array}
\label{7}%
\end{equation}
where $f_{t}:L\rightarrow\mathbb{R}$, $t\in\lbrack0,\varepsilon)$ is a smooth
family of functions and $v_{k}$ are defined by
\[%
\begin{array}
[c]{c}%
v_{k}=JF_{k}=J_{\alpha}^{\beta}F_{k}^{\alpha}\frac{\partial}{\partial
y^{\beta}}=v_{k}^{\beta}\frac{\partial}{\partial y^{\beta}}.
\end{array}
\]

\begin{definition}
If $F_{0}:L_{0}\rightarrow(M,\lambda,g,\mathbf{T},J)$ is a Legendrian
immersion such that the mean curvature form $H$ is exact, then any function
$\alpha_{0}$ with $H_{0}=d\alpha_{0}$ is called the Legendrian angle of
$F_{0}(L_{0})$.
\end{definition}

Now we recall the following definitions.

\begin{definition}
\label{d3} A Sasakian manifold $M^{2n+1}$ is called $\eta$-Einstein if there
is a constant $K$ such that the Ricci curvature
\[
Ric=Kg+(2n-K)\eta\otimes\eta
\]
and then
\[
Ric^{T}=(K+2)g^{T}%
\]
which is also called the transverse K\"{a}hler-Einstein or pseudo-Einstein.
\end{definition}

\begin{lemma}
(\cite[Corollary 3.3.]{s2}) If $F_{t}:L\rightarrow(M,\lambda,g,\mathbf{T},J)$,
$t\in\lbrack0,\varepsilon)$ is a smooth family of Legendrian immersions into
an $\eta$-Einstein Sasakian manifold that evolves according to (\ref{1}), then
the mean curvature form $H$ satisfies
\begin{equation}%
\begin{array}
[c]{c}%
\frac{\partial}{\partial t}H_{i}=-\nabla_{i}(\Delta f+(K+2)f),
\end{array}
\label{21}%
\end{equation}
and the cohomology class of $H$ is fixed. Moreover, if $H$ is exact at $t=0$,
then under the deformation (\ref{1}), there exists a smooth family of
functions $\alpha$ on $L$, smoothly depending on $t\in\lbrack0,\varepsilon)$
such that
\begin{equation}%
\begin{array}
[c]{c}%
\frac{\partial}{\partial t}\alpha=-(\Delta f+(K+2)f)
\end{array}
\label{24}%
\end{equation}
and
\begin{equation}
H=d\alpha\label{25}%
\end{equation}
for all $t$ and $\alpha$ is unique up to adding a function depending only on
$t$.
\end{lemma}

\begin{definition}
\label{6} Let $L_{0}$ be a Legendrian submanifold of dimension $n$ in an
$\eta$-Einstein Sasakian $(2n+1)$-manifold $(M,\lambda,g,\mathbf{T},J)$. Then
the Legendrian mean curvature flow is the solution of
\begin{equation}%
\begin{array}
[c]{c}%
\frac{d}{dt}F_{t}=-\nabla^{k}\alpha v_{k}-2\alpha\mathbf{T}%
\end{array}
\label{26}%
\end{equation}
where $\alpha$ is the Legendrian angle of $L_{t}=F_{t}(L)$. Then the parabolic
flow (\ref{1}) turns into the following
\begin{equation}%
\begin{array}
[c]{c}%
\frac{d}{dt}F_{t}=\mathbf{H}-2\alpha\mathbf{T}%
\end{array}
\label{27}%
\end{equation}
with the Legendrian mean curvature vector
\[%
\begin{array}
[c]{c}%
\mathbf{H}=-\nabla^{k}\alpha v_{k}.
\end{array}
\]

\end{definition}

\section{Convergence of Legendrian mean curvature flow}

In this section, we will prove the following Theorem \ref{22} and Theorem
\ref{Thm 3.2}, we have the Theorem \ref{Thm 3.1}. As consequences, we have
Theorem \ref{23} and Theorem \ref{Thm 3.1}, respectively.

Before we prove the main theorems, we will prove the long-time behavior of
solutions of the Legendrian mean curvature flow. We recall the following
Harnack inequality for the positive solution for the heat equation on a
compact manifold in \cite[Theorem 2.1.]{cao}.

\begin{proposition}
\label{Harnack} Let $L$ be a compact manifold of dimension $n$ and let
$g_{ij}(t),$ $0\leq t<\infty,$ be a family of Riemannian metrics on $L$ with
the following properties:%
\[%
\begin{array}
[c]{ll}%
\mathrm{(a)} & C_{1}g_{ij}(0)\leq g_{ij}(t)\leq C_{2}g_{ij}(0),\\
\mathrm{(b)} & |\frac{\partial}{\partial t}g_{ij}(t)|\leq C_{3}g_{ij}(0),\\
\mathrm{(c)} & R_{ij}(t)\geq-C_{4}g_{ij}(0),
\end{array}
\]
where $C_{1},$ $C_{2},$ $C_{3}$ and $C_{4}$ are positive constants independent
of $t$. Let $\Delta_{t}$ denotes the Laplace-Beltrami operator of the metric
$g_{ij}(t)$ and $\varphi(x,t)$ is a positive solution for the heat equation
\begin{equation}%
\begin{array}
[c]{c}%
(\frac{\partial}{\partial t}-\Delta_{t})\varphi(x,t)=0
\end{array}
\label{he}%
\end{equation}
on $L\times\lbrack0,\infty)$, then for any $\alpha>1$ we have%
\[%
\begin{array}
[c]{c}%
\underset{x\in L}{\sup}\varphi(x,t_{1})\leq\underset{x\in L}{\inf}%
\varphi(x,t_{2})(\frac{t_{2}}{t_{1}})^{\frac{n}{2}}\exp\{ \frac{C_{2}^{2}%
d^{2}}{4(t_{2}-t_{1})}+[\frac{n\alpha C_{4}}{2(\alpha-1)}+C_{2}C_{3}%
(n+E)](t_{2}-t_{1})\},
\end{array}
\]
here $d$ is the diameter of $L$ measured by the metric $g_{ij}(0)$,
$E=\sup|\nabla^{2}\log\varphi|$ and $0<t_{1}<t_{2}<\infty$.
\end{proposition}

By the Gauss equation (\ref{G}), the Ricci curvature are given by
\[%
\begin{array}
[c]{c}%
R_{ij}=Kg_{ij}+H_{m}h_{ij}^{m}-h_{il}^{m}h_{jm}^{l}%
\end{array}
\]
and the evolution equation for the metrics $g_{ij}(x,t)$ is $\frac{\partial
}{\partial t}g_{ij}=-2H^{k}h_{kij},$ then the conditions in Proposition
\ref{Harnack} are satisfied if $L_{t}$ is a family of Legendrian immersions
into an $\eta$-Einstein Sasakian manifold $M^{2n+1}$ evolving by the mean
curvature flow (\ref{27}) such that the second fundamental forms satisfy
$|A|^{2}\leq C_{0}$ uniformly in $t$ and all metrics $g_{ij}(x,t)$ are
uniformly equivalent to $g_{ij}(x,0)$. On the other hand, the Legendrian angle
$\alpha$ satisfies the equation%
\begin{equation}%
\begin{array}
[c]{c}%
\frac{\partial}{\partial t}\alpha=\Delta\alpha+(K+2)\alpha,
\end{array}
\label{angle}%
\end{equation}
the strong parabolic maximum principle implies that $\alpha$ is either
identical zero or $\alpha-\inf\alpha(0)$ is a positive solution. Following the
same methods as in \cite{cao}, we can obtain the exponential decreasing of the
oscillation of the rescaled Legendrian angle function $\beta:=e^{-(K+2)t}%
\alpha,$ which satisfies the heat equation (\ref{he}).

\begin{lemma}
\label{osc} (\cite{s3}) Let $osc\beta(t):=\sup\beta(t)-\inf\beta(t)$ be the
oscillation of $\beta$, where $\beta$ is the solution of the heat equation
(\ref{he}). If $L_{t}$ is a family of compact Legendrian immersions in an
$\eta$-Einstein Sasakian manifold $(M^{2n+1},J,T)$ evolving by the mean
curvature flow (\ref{27}) such that the norm of the second fundamental forms
on $L_{t}$ are uniformly bounded and all induced metrics $g_{ij}(t)$ are
uniformly equivalent to $g_{ij}(0)$, then there exist positive constants
$C_{5}$ and $C_{6}$ independent of $t$ such that%
\begin{equation}%
\begin{array}
[c]{c}%
osc\beta(t)\leq C_{5}e^{-C_{6}t},\text{\ for all\ }t\in\lbrack0,\infty).
\end{array}
\label{101}%
\end{equation}

\end{lemma}

{\large Proof of Theorem \ref{12}: }

\begin{proof}
As in \cite{Hu}, we can use the evolution equation for the second fundamental
tensor $|A|^{2}$ to derive the following inequalities for all $m\geq0$:%
\begin{equation}%
\begin{array}
[c]{lll}%
\frac{\partial}{\partial t}|\nabla^{m}A|^{2} & \leq & \Delta|\nabla^{m}%
A|^{2}-2|\nabla^{m+1}A|^{2}\\
&  & +C(n,m)\left\{  \underset{i+j+k=m}{%
%TCIMACRO{\tsum }%
%BeginExpansion
{\textstyle\sum}
%EndExpansion
}|\nabla^{i}A||\nabla^{j}A||\nabla^{k}A||\nabla^{m}A|\right. \\
&  & \left.  \text{ \ \ \ \ \ \ \ \ \ }+\widetilde{C}_{m}\underset{i\leq m}{%
%TCIMACRO{\tsum }%
%BeginExpansion
{\textstyle\sum}
%EndExpansion
}|\nabla^{i}A||\nabla^{m}A|+\widetilde{C}_{m+1}|\nabla^{m}A|\right\}  ,
\end{array}
\label{100}%
\end{equation}
where the constants depend only on $n$ and $M$. Assume that $m\geq0$ and that
we have bounds for $|\nabla^{l}A|^{2}$ for all $0\leq l<m$. Then by using the
maximum principle, we can prove that the uniform bounds for $|\nabla^{m}%
A|^{2}$, for all $m\geq0$. Hence we proved part $(1)$. Part $(2) $ is now a
direct consequence of $(1)$ since we can continue the solution of the flow
(\ref{27}) if $|A|^{2}$ will stay uniformly bounded on a finite interval. In
part $(3),$ in the case that the ambient $\eta$-Einstein Sasakian manifold
with the $\eta$-Einstein constant $K+2>0$ (which is Sasaki-Einstein). It
follows from the result of Chen-He \cite{ch} easily. To prove the part $(3)$
in case of the $\eta$-Einstein constant $K+2\leq0$, we observe that in this
case all the assumptions in Lemma \ref{osc} are satisfied and one then obtain
from (\ref{101}) that
\[%
\begin{array}
[c]{c}%
osc(e^{-(K+2)t}\alpha)\leq ce^{-\delta t}%
\end{array}
\]
for some positive constants $c$, $\delta$ that do not depend on $t$. This
implies that%
\begin{equation}%
\begin{array}
[c]{c}%
osc(\alpha(t))\leq ce^{(K+2-\delta)t}%
\end{array}
\label{102}%
\end{equation}
and hence the Legendrian angle $\alpha(t)$ tends to zero if $K+2\leq0$ and
$t\rightarrow\infty$. Since by the maximum principle $\alpha(t)$ is uniformly
bounded if $K+2\leq0,$ and we see that%
\[%
\begin{array}
[c]{c}%
|\alpha-\inf\alpha(0)|\leq2ce^{(K+2-\delta)t}.
\end{array}
\]
For any compact Riemannian manifold $(L,g)$, there exists a constant $C$
depending only on $L$ and $g$ such that the interpolation inequality
\[%
\begin{array}
[c]{c}%
|\nabla\varphi|^{2}\leq C|\varphi|(|\nabla^{2}\varphi|+|\nabla\varphi|)
\end{array}
\]
holds for any smooth functions $\varphi$. Now, if $g(t)$ is a family of
uniformly equivalent metrics on $L$, then we can choose a positive constant
$C$ independent of $t.$ By applying the above equation to $\varphi=\alpha
-\inf\alpha(0)$, note $H=d\alpha,$ and we deduce that
\[%
\begin{array}
[c]{c}%
|H|^{2}\leq C|\alpha-\inf\alpha(0)|(|\nabla H|+|H|)\leq C^{\prime
}e^{(K+2-\delta)t}%
\end{array}
\]
since $|\nabla H|^{2}\leq c(n)|\nabla A|^{2}\leq C(n)$ by (a). This implies
the mean curvature vector $H$ tends to zero exponentially as $t$ goes
infinite. We then integrate the evolution equation $\frac{d}{dt}F=H-2\alpha T
$ and the exponential decay of $|H|$ and $\alpha$ implies that for any
$\gamma>0$ we can find a time $t_{0}$ such that for all $t\geq t_{0}$ the
immersion $L_{t}$ will stay in an $\gamma$-neighborhood of $L_{t_{0}}$. This
proves the convergence as $t$ goes infinite.
\end{proof}

Before we prove main theorems, we recall the following definition.

\begin{definition}
(\cite{p}) A geodesic ball $B(p,\rho)\subset L$ is called $\kappa
$-noncollapsed if
\[%
\begin{array}
[c]{c}%
\mathrm{Vol}(B(q,s))\geq\kappa s^{n},
\end{array}
\]
whenever $B(q,s)\subset B(p,\rho)$. Here the volume is with respect to the
induced metric on $L$. A Riemannian manifold $L$ is called $\kappa
$-noncollapsed on the scale $r$ if every geodesic ball $B(q,s)$ is $\kappa
$-noncollapsed for $s\leq r$.
\end{definition}

\begin{theorem}
\label{22} Let $(M,\lambda,g,\mathbf{T},J)$ be an $\eta$-Einstein Sasakian
manifold with the $\eta$-Einstein constant $K+2<0$ and the solution $L_{t}$
$(t\in\lbrack0,T_{0}])$ of the Legendrian mean curvature flow with exact mean
curvature form $H_{0}$ for $t=0$. For any $\kappa_{0},$ $r_{0},$ $V_{0},$
$\Lambda_{0}>0$, there exists $\epsilon_{0}=\epsilon_{0}(\kappa_{0}%
,r_{0},V_{0},\Lambda_{0},K_{5},K+2)$ if $L_{0}$ is $\kappa_{0}$-noncollapsed
on the scale $r_{0}$ and satisfies
\[%
\begin{array}
[c]{c}%
\text{\textrm{Vol}}(L_{0})\leq V_{0},\text{ }|A|\leq\Lambda_{0},\text{
}|H|\leq\epsilon_{0},
\end{array}
\]
then the Legendrian mean curvature flow with the initial data $L_{0}$ will
converge exponentially fast to a minimal Legendrian submanifold in $M$.
\end{theorem}

In order to proof the Theorem \ref{22}, we need the following Lemmas.

\begin{lemma}
\label{11} Let $(M,\lambda,g,\mathbf{T},J)$ be an $\eta$-Einstein Sasakian
manifold with the $\eta$-Einstein constant $K+2$. For any positive constants
$\Lambda,$ $\epsilon>0$, if the solution $L_{t}$ $(t\in\lbrack0,T_{0}])$ of
the Legendrian mean curvature flow with the exact mean curvature form $H_{0}$
for $t=0$ satisfies
\begin{equation}%
\begin{array}
[c]{c}%
|A|(t)\leq\Lambda,\text{\ }|H|(t)\leq\epsilon,\text{\ }t\in\lbrack0,T_{0}]
\end{array}
\label{41}%
\end{equation}
then we have the inequality
\begin{equation}%
\begin{array}
[c]{c}%
\frac{d}{dt}\int_{L_{t}}|H|^{2}d\mu_{t}\leq2\left(  K+2+\Lambda\epsilon
\right)  \int_{L_{t}}|H|^{2}d\mu_{t},\text{\ }t\in\lbrack0,T_{0}].
\end{array}
\label{42}%
\end{equation}
Moreover, we have
\begin{equation}%
\begin{array}
[c]{c}%
\frac{d}{dt}\int_{L_{t}}|H|^{2}d\mu_{t}\leq-2\left(  \lambda_{1}%
-(K+2)-\Lambda\epsilon\right)  \int_{L_{t}}|H|^{2}d\mu_{t},\text{\ }%
t\in\lbrack0,T_{0}].
\end{array}
\label{43}%
\end{equation}
where $\lambda_{1}$ is the first eigenvalue of Laplacian $\Delta$ with respect
to the induced metric on $L_{t}$.
\end{lemma}

\begin{proof}
By using equation (\ref{i}) one get $\frac{\partial}{\partial t}g^{_{ij}%
}=2H_{k}h^{kij}$ and from (\ref{21}) we have $\frac{\partial}{\partial t}%
H_{i}=\nabla_{i}\Delta\alpha+(K+2)H_{i}$ for a smooth Legendrian angle
function $\alpha$,
\begin{equation}%
\begin{array}
[c]{ll}
& \frac{d}{dt}\int_{L_{t}}|H|^{2}d\mu_{t}=\frac{d}{dt}\int_{L_{t}}g^{ij}%
H_{i}H_{j}d\mu_{t}\\
= & \int_{L_{t}}[(\frac{\partial}{\partial t}g^{ij})H_{i}H_{j}+2H^{i}%
\frac{\partial}{\partial t}H_{i}-|H|^{4}]d\mu_{t}\\
= & \int_{L_{t}}[2H^{i}\nabla_{i}\Delta\alpha+2(K+2)|H|^{2}+2H_{i}H_{j}%
H_{k}h^{kij}-|H|^{4}]d\mu_{t}\\
= & \int_{L_{t}}[-2(\Delta\alpha)^{2}+2(K+2)|H|^{2}+2H_{i}H_{j}H_{k}%
h^{kij}-|H|^{4}]d\mu_{t},
\end{array}
\label{44}%
\end{equation}
which implies
\[%
\begin{array}
[c]{c}%
\frac{d}{dt}\int_{L_{t}}|H|^{2}d\mu_{t}\leq2\left(  K+2+\Lambda\epsilon
\right)  \int_{L_{t}}|H|^{2}d\mu_{t}.
\end{array}
\]

From (\ref{25}) one have that $H=\nabla\alpha$ and yields%
\[%
\begin{array}
[c]{c}%
\int_{L_{t}}(\Delta\alpha)^{2}d\mu_{t}\geq\lambda_{1}\int_{L_{t}}|\nabla
\alpha|^{2}d\mu_{t}=\lambda_{1}\int_{L_{t}}|H|^{2}d\mu_{t},
\end{array}
\]
combining this with (\ref{44}), we have the inequality (\ref{43}).
\end{proof}

\begin{lemma}
\label{8} Let $(M,\lambda,g,\mathbf{T},J)$ be an $\eta$-Einstein Sasakian
manifold with the $\eta$-Einstein constant $K+2$ and the solution $L_{t}$
$(t\in\lbrack0,T_{0}])$ of the Legendrian mean curvature flow with exact mean
curvature form $H_{0}$ for $t=0$. If $L_{0}$ is $\kappa_{0}$-noncollapsed on
the scale $r_{0}$, then for any small geodesic ball $B_{t}(p,\rho)$ in $L_{t}$
with radius $\rho\in(o,r_{0})$, we have
\[%
\begin{array}
[c]{c}%
\mathrm{Vol}(B_{t}(p,\rho))\geq\kappa_{0}e^{-(n+1)E(t)}\rho^{n},
\end{array}
\]
where $E(t)=\int_{0}^{t}\max_{L_{s}}(|A||H|+|H|^{2})ds.$
\end{lemma}

\begin{proof}
Consider a fixed smooth path $\gamma:[a,b]\rightarrow L_{t}$ with
$\gamma(a)=p$ and $\gamma(b)=q$ whose length at time $t$ is given by
$d_{t}(p,q)=\int_{a}^{b}\left\vert \gamma^{\prime}(t)\right\vert _{g(t)}dr,$
where $r$ is the arc length. Then, by (\ref{i}) $\frac{\partial}{\partial
t}g^{_{ij}}=2H_{k}h^{kij},$ we compute%
\[%
\begin{array}
[c]{c}%
\frac{\partial}{\partial t}d_{t}(p,q)=\frac{1}{2}\int_{\gamma}(\frac{\partial
}{\partial t}g^{ij})(\xi_{i},\gamma_{j}^{\prime})dr=\int_{\gamma}H_{k}%
h^{kij}\xi_{i}\gamma_{j}^{\prime}dr\geq-(\max_{L_{t}}\left\vert A\right\vert
|H|)d_{t}(p,q)
\end{array}
\]
here $\xi=\gamma^{\prime}/\left\vert \gamma^{\prime}\right\vert $ is the unit
tangent vector to the path $\gamma.$ So we obtain that the distance function
$d_{t}$ will satisfy
\[%
\begin{array}
[c]{c}%
d_{t}(p,q)\geq\exp(-\int_{0}^{t}\max_{L_{s}}\left\vert A\right\vert
|H|ds)d_{0}(p,q)\geq e^{-E(t)}d_{0}(p,q).
\end{array}
\]
Also from (\ref{V}) the volume form $d\mu_{t}$ satisfies $\frac{\partial
}{\partial t}d\mu_{t}=-\left\vert H\right\vert ^{2}d\mu_{t}$, which implies
\[%
\begin{array}
[c]{c}%
d\mu_{t}\geq\exp(-\int_{0}^{t}\max_{L_{s}}|H|^{2}ds)d\mu_{0}\geq e^{-E(t)}%
d\mu_{0}.
\end{array}
\]
Thus, the volume of the geodesic ball $B_{t}(p,\rho)$ has the following lower
bound%
\[%
\begin{array}
[c]{c}%
\mathrm{Vol}(B_{t}(p,\rho))=\int_{B_{t}(p,\rho)}d\mu_{t}\geq\int
_{B_{0}(p,e^{-E(t)}\rho)}e^{-E(t)}d\mu_{0}\geq\kappa_{0}e^{-(n+1)E(t)}\rho
^{n},
\end{array}
\]
for $\rho\leq r_{0}$, since $L_{0}$ is $\kappa_{0}$-noncollapsed on the scale
$r_{0}$.
\end{proof}

\begin{lemma}
\label{13} (\cite{li}) Suppose that $L$ is $\kappa$-noncollapsed on the scalar
$r$. For any tensor $S$ on $L$, if
\[%
\begin{array}
[c]{c}%
|\nabla S|\leq\Lambda,\text{\ }\int_{L}|S|^{2}d\mu\leq\epsilon,
\end{array}
\]
where $\epsilon\leq r^{n+2}$, then
\[%
\begin{array}
[c]{c}%
\max_{L}|S|\leq(\frac{1}{\sqrt{\kappa}}+\Lambda)\epsilon^{\frac{1}{n+2}}.
\end{array}
\]

\end{lemma}

\begin{lemma}
\label{9} Let $(M,\lambda,g,\mathbf{T},J)$ be an $\eta$-Einstein Sasakian
manifold with the $\eta$-Einstein constant $K+2$ and the solution $L_{t}$
$(t\in\lbrack0,T_{0}])$ of the Legendrian mean curvature flow with exact mean
curvature form $H_{0}$ for $t=0$. Suppose that
\[
|A|(0)\leq\Lambda,\text{\ }|H|(0)\leq\epsilon,
\]
then we have
\[
|A|(t)\leq2\Lambda,\text{\ }|H|(t)\leq2\epsilon,\text{\ }t\in\lbrack0,T_{1}],
\]
where $T_{1}=T_{1}(n,\Lambda,K_{0},K_{1})$.
\end{lemma}

\begin{proof}
By using the evolution equation (\ref{A}) of the second fundamental form and
the assumption $K_{i}=%
%TCIMACRO{\tsum \limits_{k=0}^{i}}%
%BeginExpansion
{\textstyle\sum\limits_{k=0}^{i}}
%EndExpansion
\sup|\bigtriangledown^{k}\overline{Rm}|$, we have%
\begin{equation}%
\begin{array}
[c]{c}%
\frac{\partial}{\partial t}|A|\leq\Delta|A|+c_{1}(n)|A|^{3}+c_{2}%
(n,K_{0})|A|+c_{3}(n,K_{1}).
\end{array}
\label{51}%
\end{equation}
Let $t_{0}=\sup\left\{  s>0:|A|(t)\leq2\Lambda,\text{\ }t\in\lbrack
0,s)\right\}  $. Then for $t\in\lbrack0,t_{0})$ we have%
\[%
\begin{array}
[c]{c}%
\frac{\partial}{\partial t}|A|\leq\Delta|A|+8c_{1}\Lambda^{3}+2c_{2}%
\Lambda+c_{3}.
\end{array}
\]
By using the maximum principle, we have
\[%
\begin{array}
[c]{c}%
|A|(t)\leq\max_{L_{0}}|A|(0)+(8c_{1}\Lambda^{3}+2c_{2}\Lambda+c_{3}%
)t\leq2\Lambda,\text{ }t\in\lbrack0,\frac{\Lambda}{8c_{1}\Lambda^{3}%
+2c_{2}\Lambda+c_{3}}].
\end{array}
\]

In the following, we give the estimate for the mean curvature vector $H$. For
$t\in\lbrack0,t_{0}]$, use the evolution equation (\ref{H}) for $H,$ we have
\[%
\begin{array}
[c]{c}%
\frac{\partial}{\partial t}|H|\leq\Delta|H|+|A|^{2}|H|+2|H|\leq\Delta
|H|+(4\Lambda^{2}+2)|H|,
\end{array}
\]
which implies%
\[%
\begin{array}
[c]{c}%
|H|(t)\leq|H|(0)e^{(4\Lambda^{2}+2)t}\leq2\epsilon,\text{\ }t\in\lbrack
0,\min\{t_{0},\frac{\log2}{4\Lambda^{2}+2}\}].
\end{array}
\]
Thus%
\[%
\begin{array}
[c]{c}%
|A|(t)\leq2\Lambda,\text{\ }|H|(t)\leq2\epsilon,\text{ }t\in\lbrack0,T_{1}]
\end{array}
\]
with $T_{1}=\min\{ \frac{\Lambda}{8c_{1}\Lambda^{3}+2c_{2}\Lambda+c_{3}}%
,\frac{\log2}{4\Lambda^{2}+2}\}.$
\end{proof}

In order to prove our main theorem, we introduce the following definition.

\begin{definition}
Let $(M,\lambda,g,\mathbf{T},J)$ be an $\eta$-Einstein Sasakian manifold with
the $\eta$-Einstein constant $K+2<0$. For any positive constants
$\kappa,r,\Lambda,\epsilon$, we define the following subspace of Legendrian
submanifolds in $M^{2n+1}$ by
\begin{equation}%
\begin{array}
[c]{ccc}%
\mathcal{A}(\kappa,r,\Lambda,\epsilon) & = & \left\{  L:L\text{ is }%
\kappa\text{-noncollapsed on the scale }r\text{ }\right. \\
&  & \left.  \text{ \ \ \ \ \ \ \ \ \ \ \ \ \ \ \ \ \ \ \ \ with }%
|A|(t)\leq\Lambda,\text{ }|H|(t)\leq\epsilon\right\}
\end{array}
\label{LS}%
\end{equation}

\end{definition}

By using Lemma \ref{8} and Lemma \ref{9}, we know that the following Lemma is true.

\begin{lemma}
\label{10} Let $(M,\lambda,g,\mathbf{T},J)$ be an $\eta$-Einstein Sasakian
manifold with the $\eta$-Einstein constant $K+2$ and the solution $L_{t}$
$(t\in\lbrack0,T_{0}])$ of the Legendrian mean curvature flow with exact mean
curvature form $H_{0}$ for $t=0$. If the initial Legendrian submanifold
$L_{0}\in\mathcal{A}(\kappa,r,\Lambda,\epsilon)$, then there exists $\tau
=\tau(n,\Lambda,K_{0},K_{1})>0$, such that
\[%
\begin{array}
[c]{c}%
L_{t}\in\mathcal{A}(\frac{1}{2}\kappa,r,2\Lambda,2\epsilon)\text{\ for }%
t\in\lbrack0,\tau].
\end{array}
\]

\end{lemma}

The following Lemma is similar to (\cite{li}, Lemma 4.2). For the
completeness, we sketch the proof here.

\begin{lemma}
\label{16}Let $(M,\lambda,g,\mathbf{T},J)$ be an $\eta$-Einstein Sasakian
manifold with the $\eta$-Einstein constant $K+2<0$ and the solution $L_{t}$
$(t\in\lbrack0,T_{0}])$ of the Legendrian mean curvature flow with exact mean
curvature form $H_{0}$ for $t=0$. For any $\kappa_{0},$ $r_{0},$ $\Lambda
_{0},$ $V_{0},$ $T_{0}>0$, assume that there exists a constant $\epsilon
_{0}=\epsilon_{0}(\kappa_{0},r_{0},\Lambda_{0},V_{0})$ such that
\[%
\begin{array}
[c]{l}%
1.\text{ }L_{0}\in\mathcal{A}(\kappa_{0},r_{0},\Lambda_{0},\epsilon
_{0})\text{\ and \textrm{Vol}}(L_{0})\leq V_{0},\\
2.\text{ }L_{t}\in\mathcal{A}(\frac{1}{3}\kappa_{0},r_{0},6\Lambda
_{0},2\epsilon_{0}^{\frac{1}{n+2}}),\text{ }(t\in\lbrack0,T_{0}]).
\end{array}
\]
Then we have\newline(1) The mean curvature vector satisfies
\begin{equation}%
\begin{array}
[c]{c}%
\max_{L_{t}}|H|(t)\leq\epsilon_{0}^{\frac{1}{n+2}}e^{\frac{K+2}{n+2}%
t},\text{\ }t\in\lbrack\tau,T_{0}].
\end{array}
\label{60}%
\end{equation}
(2) The second fundamental form
\[%
\begin{array}
[c]{c}%
\max_{L_{t}}|A|(t)\leq3\Lambda_{0},\text{ }t\in\lbrack0,T_{0}].
\end{array}
\]
(3) $L_{t}$ is $\frac{2}{3}\kappa_{0}$-noncollapsed on the scale $r_{0}$ for
$t\in\lbrack0,T_{0}]$.\newline Thus, the solution $L_{t}\in\mathcal{A}%
(\frac{2}{3}\kappa_{0},r_{0},3\Lambda_{0},\epsilon_{0}^{\frac{1}{n+2}})$ for
$t\in\lbrack0,T_{0}]$, and by using Lemma \ref{10}, we obtain that
\[%
\begin{array}
[c]{c}%
L_{t}\in\mathcal{A}(\frac{1}{3}\kappa_{0},r_{0},6\Lambda_{0},2\epsilon
_{0}^{\frac{1}{n+2}}),\text{ }t\in\lbrack0,T_{0}+\delta]
\end{array}
\]
for some\ $\delta=\delta(n,\Lambda_{0},K_{1})>0.$
\end{lemma}

\begin{proof}
(1) For any $\Lambda_{0}>0$, we can choose $\epsilon_{0}$ small enough such
that $24\Lambda_{0}\epsilon_{0}^{\frac{1}{n+2}}\leq-(K+2).$ By using the
inequality (\ref{42}) in Lemma \ref{11}, we have%
\begin{equation}%
\begin{array}
[c]{lll}%
\int_{L_{t}}|H|^{2}d\mu_{t} & \leq & e^{2(K+2+12\Lambda_{0}\epsilon_{0}%
^{\frac{1}{n+2}})t}\int_{L_{0}}|H|^{2}d\mu_{0}\\
& \leq & e^{(K+2)t}\int_{L_{0}}|H|^{2}d\mu_{0}\\
& \leq & V_{0}\epsilon_{0}^{2}e^{(K+2)t},\text{ }t\in\lbrack0,T_{0}].
\end{array}
\label{61}%
\end{equation}
By the fact $L_{t}\in\mathcal{A}(\frac{1}{3}\kappa_{0},r_{0},6\Lambda
_{0},2\epsilon_{0}^{\frac{1}{n+2}}),$ $(t\in\lbrack0,T_{0}])$ and Theorem
\ref{12}, there exists a constant $C_{1}=C_{1}(n,\Lambda_{0},K_{2})$ such
that
\begin{equation}%
\begin{array}
[c]{c}%
|\nabla A|(t)\leq C_{1}(n,\Lambda_{0},K_{2},\tau),\text{\ }t\in\lbrack
\tau,T_{0}],
\end{array}
\label{62}%
\end{equation}
where $\tau=\tau(n,\Lambda_{0},K_{1})$ as in Lemma \ref{10}. By using Lemma
\ref{13}, (\ref{61}) and (\ref{62}), we have
\[%
\begin{array}
[c]{c}%
|H|(t)\leq(\frac{3}{\sqrt{\kappa_{0}}}+C_{1})\left(  V_{0}\epsilon_{0}%
^{2}e^{(K+2)t}\right)  ^{\frac{1}{n+2}},\text{\ }t\in\lbrack\tau,T_{0}],
\end{array}
\]
where we have used that fact that $L_{t}$ is $\frac{1}{3}\kappa_{0}%
$-noncollapsed on the scale $r_{0}$ and $V_{0}\epsilon_{0}^{2}\leq r_{0}%
^{n+2}$ if $\epsilon_{0}$ is small enough. Set $(\frac{3}{\sqrt{\kappa_{0}}%
}+C_{1})\left(  V_{0}\epsilon_{0}\right)  ^{\frac{1}{n+2}}\leq1$ for
$\epsilon_{0}$ small enough, we have
\[%
\begin{array}
[c]{c}%
|H|(t)\leq\epsilon_{0}^{\frac{1}{n+2}}e^{\frac{K+2}{n+2}t},\ t\in\lbrack
\tau,T_{0}].
\end{array}
\]
(2) By Theorem \ref{12}, there exist some constants $C_{k}=C_{k}(n,\Lambda
_{0},K_{k+1})$ such that%
\[
|\nabla^{k}A|(t)\leq C_{k}(n,\Lambda_{0},K_{k+1},\tau),\ t\in\lbrack\tau
,T_{0}].
\]
So by property (a) and integration by parts, we can obtain%
\begin{equation}%
\begin{array}
[c]{l}%
\int_{L_{t}}|\nabla^{2}H|^{2}d\mu_{t}\leq\int_{L_{t}}|H||\nabla^{4}H|d\mu
_{t}\leq C_{4}V_{0}\epsilon_{0}^{\frac{1}{n+2}}e^{\frac{K+2}{n+2}t},\text{
}t\in\lbrack\tau,T_{0}].
\end{array}
\label{63}%
\end{equation}
where we used the fact that $\mathrm{Vol}(L_{t})$ is decreasing along the
flow. By Lemma \ref{13} again, we have%
\begin{equation}%
\begin{array}
[c]{l}%
|\nabla^{2}H|^{2}(t)\leq(\frac{3}{\sqrt{\kappa_{0}}}+C_{3})(C_{4}V_{0}%
\epsilon_{0}^{\frac{1}{n+2}}e^{\frac{K+2}{n+2}t})^{\frac{1}{n+2}}.
\end{array}
\label{64}%
\end{equation}
By the equation (\ref{3}) with $H=-\nabla f$, we have%
\begin{equation}%
\begin{array}
[c]{c}%
\frac{\partial}{\partial t}|A|\leq|\nabla^{2}H|+c(n)|A|^{2}|H|+(|R_{\alpha
\beta\gamma\delta}|+3)|H|,
\end{array}
\label{65}%
\end{equation}
So we have%
\[%
\begin{array}
[c]{lll}%
|A|(t) & \leq & |A|(\tau)+\int_{\tau}^{t}|\nabla^{2}H|+(c(n)|A|^{2}%
+K_{0})|H|\\
& \leq & 2\Lambda_{0}+(\frac{3}{\sqrt{\kappa_{0}}}+C_{3})(C_{4}V_{0}%
\epsilon_{0}^{\frac{1}{n+2}})^{\frac{1}{n+2}}\frac{(n+2)^{2}}{|K+2|}\\
&  & +\left(  K_{0}+36c(n)\Lambda_{0}^{2}\right)  \epsilon_{0}^{\frac{1}{n+2}%
}\frac{n+2}{|K+2|}\\
& \leq & 3\Lambda_{0},
\end{array}
\]
for $\epsilon_{0}$ small enough.

(3) By using the definition of $E(t)$ and Lemma \ref{10}, we have%
\begin{equation}%
\begin{array}
[c]{lll}%
E(t) & \leq & \int_{0}^{\tau}\max_{L_{s}}(|A||H|+|H|^{2})ds+\int_{\tau}%
^{t}\max_{L_{s}}(|A||H|+|H|^{2})ds\\
& \leq & 4\Lambda_{0}\epsilon_{0}\tau+4\epsilon_{0}^{2}\tau+(3\Lambda
_{0}\epsilon_{0}^{\frac{1}{n+2}}+\epsilon_{0}^{\frac{2}{n+2}})\frac
{n+2}{|K+2|}\\
& \leq & \frac{1}{n+1}\log\frac{3}{2},
\end{array}
\label{66}%
\end{equation}
for $t\in\lbrack0,T_{0}]$ and $\epsilon_{0}$ is small enough. By Lemma
\ref{10}, we know that $L_{t}$ is $\frac{2}{3}\kappa_{0}$-noncollapsed on the
scale $r_{0}$ for $t\in\lbrack0,T_{0}]$.
\end{proof}

{\large Proof of Theorem \ref{22}: }

\begin{proof}
Suppose that $L_{0}\in\mathcal{A}(\kappa_{0},r_{0},\Lambda_{0},\epsilon_{0})$
for any positive constants $\kappa_{0},$ $r_{0},$ $\Lambda_{0}$ and
$\epsilon_{0}$ small enough. Set%
\[%
\begin{array}
[c]{c}%
t_{0}=\sup\{t>0|\text{ }L_{s}\in\mathcal{A}(\frac{1}{3}\kappa_{0}%
,r_{0},6\Lambda_{0},2\epsilon_{0}^{\frac{1}{n+2}}),\text{ }s\in\lbrack0,t)\}.
\end{array}
\]
In the following, we will show that $t_{0}=\infty$. Suppose that $t_{0}%
<\infty$. By Lemma \ref{16}, we know that%
\[%
\begin{array}
[c]{c}%
L_{t}\in\mathcal{A}(\frac{2}{3}\kappa_{0},r_{0},3\Lambda_{0},\epsilon
_{0}^{\frac{1}{n+2}}),
\end{array}
\]
for $t\in\lbrack0,t_{0})$ and $\epsilon_{0}=\epsilon_{0}(\kappa_{0}%
,r_{0},\Lambda_{0},V_{0},n,K_{5})>0$. By using Lemma \ref{16} again, the
solution $L_{t}$ can be extended to $[0,t_{0}+\delta]$ and $L_{t}$ is
contained in $\mathcal{A}(\frac{2}{3}\kappa_{0},r_{0},3\Lambda_{0}%
,\epsilon_{0}^{\frac{1}{n+2}})$ which contradicts the definition of $t_{0}$.
So we know that $t_{0}=\infty$ and%
\[%
\begin{array}
[c]{c}%
L_{t}\in\mathcal{A}(\frac{1}{3}\kappa_{0},r_{0},6\Lambda_{0},2\epsilon
_{0}^{\frac{1}{n+2}}),\text{ }%
\end{array}
\]
for $t\in\lbrack0,\infty).$ By using (\ref{60}), we obtain that the mean
curvature vector will decay exponentially to zero and the flow will converge
to a smooth minimal Legendrian submanifold.
\end{proof}

{\large Proof of Theorem \ref{23}:}

\begin{proof}
By the conditions (\ref{70}) and Lemma \ref{9}, there exists a constant
$T_{1}=T(n,\Lambda_{0},K_{1})$ such that%
\begin{equation}
|A|(t)\leq2\Lambda_{0},\text{ }t\in\lbrack0,T_{1}].\label{71}%
\end{equation}
By using the inequality
\[
2H_{i}H_{j}H_{k}h^{kij}\leq|A|^{2}|H|^{2}+|H|^{4}\leq4\Lambda_{0}^{2}%
|H|^{2}+|H|^{4}%
\]
and the equation (\ref{44}) in the proof of Lemma \ref{11}, we have%
\[%
\begin{array}
[c]{lll}%
\frac{d}{dt}\int_{L_{t}}|H|^{2}d\mu_{t} & \leq & \int_{L_{t}}[2(K+2)|H|^{2}%
+2H_{i}H_{j}H_{k}h^{kij}-|H|^{4}]d\mu_{t}\\
& \leq & (2\left(  K+2)+4\Lambda_{0}^{2}\right)  \int_{L_{t}}|H|^{2}d\mu_{t}\\
& \leq & 4\Lambda_{0}^{2}\int_{L_{t}}|H|^{2}d\mu_{t}.
\end{array}
\]
Thus we have%
\begin{equation}%
\begin{array}
[c]{l}%
\int_{L_{t}}|H|^{2}d\mu_{t}\leq e^{4\Lambda_{0}^{2}t}\int_{L_{0}}|H|^{2}%
d\mu_{0}\leq\epsilon_{0}e^{4\Lambda_{0}^{2}t},\text{ }t\in\lbrack0,T_{1}].
\end{array}
\label{73}%
\end{equation}
So we can choose a positive constant $t_{0}=t_{0}(n,\Lambda_{0},K_{1})<T_{1} $
such that%
\begin{equation}%
\begin{array}
[c]{l}%
\int_{L_{t}}|H|^{2}d\mu_{t}\leq2\epsilon_{0},\text{ }t\in\lbrack0,t_{0}].
\end{array}
\label{74}%
\end{equation}
By using Proposition 2.2 in \cite{ch}, we can know that the injectivity radius
of $L_{t}$ is bounded from below%
\begin{equation}%
\begin{array}
[c]{c}%
\mathrm{inj}(L_{t})\geq\iota,\text{ }t\in\lbrack\frac{1}{2}t_{0},t_{0}].
\end{array}
\label{75}%
\end{equation}
for some constant $\iota=\iota(n,\Lambda_{0},K_{0},\iota_{0})$. By Gauss
equation and (\ref{71}) the intrinsic curvature of $L_{t}$ is uniformly
bounded%
\begin{equation}%
\begin{array}
[c]{c}%
|Rm|\leq C(K_{0},\Lambda),\text{ }t\in\lbrack\frac{1}{2}t_{0},t_{0}].
\end{array}
\label{76}%
\end{equation}
By (\ref{75}) and (\ref{76}) and the volume comparison theorem, there exist
$\kappa_{0}=\kappa_{0}(n,\iota_{0},K_{0},\Lambda_{0})$ and $r_{0}%
=r_{0}(n,\iota_{0},K_{0},\Lambda_{0})$ such that $L_{t}$ is $\kappa_{0}%
$-noncollapsed on the scale $r_{0}$ for all $t\in\lbrack\frac{1}{2}t_{0}%
,t_{0}].$

By (\ref{71}) and Theorem \ref{12}, we have%
\[%
\begin{array}
[c]{c}%
|\nabla A|\leq C_{1}(n,\Lambda_{0},K_{2}),\text{ }t\in\lbrack\frac{1}{2}%
t_{0},t_{0}].
\end{array}
\]
By using Lemma \ref{13}, we have%
\[%
\begin{array}
[c]{l}%
|H|(t)\leq(\frac{1}{\sqrt{\kappa_{0}}}+2C_{1})(2\epsilon_{0})^{\frac{1}{n+2}%
},\text{ }t\in\lbrack\frac{1}{2}t_{0},t_{0}].
\end{array}
\]
Applying Theorem \ref{22}, we know that Theorem \ref{23} is true.
\end{proof}

Now, we will obtain the following results for $\eta$-Einstein Sasakian
manifold with the $\eta$-Einstein constant $K+2\geq0$.

\begin{theorem}
\label{Thm 3.2}Let $(M,\lambda,g,\mathbf{T},J)$ be an $\eta$-Einstein Sasakian
manifold with the $\eta$-Einstein constant $K+2\geq0$ and $L$ be a compact
Legendrian submanifold smoothly immersed in $M$. For any $\kappa_{0},$
$r_{0},$ $V_{0},$ $\Lambda_{0},$ $\delta_{0}>0$, there exists $\epsilon
_{0}=\epsilon_{0}(\kappa_{0},r_{0},V_{0},\Lambda_{0},\delta_{0},K+2,K_{5}%
,\iota_{0})>0$ such that if $L_{0}$ satisfies
\begin{equation}%
\begin{array}
[c]{c}%
\lambda_{1}\geq K+2+\delta_{0},\text{ \textrm{Vol}}(L_{0})\leq V_{0}%
,\text{\ }|A|\leq\Lambda_{0},\text{\ }|H|^{2}\leq\epsilon_{0}.
\end{array}
\label{70A}%
\end{equation}
Then the Legendrian mean curvature flow (\ref{27}) with the exact mean
curvature form $H_{0}$ for $t=0$ will converge exponentially fast to a minimal
Legendrian in $M$.
\end{theorem}

In order to proof the Theorem \ref{Thm 3.2}, we need the following lemma which
show that the first eigenvalue of the Laplacian have a positive lower bound if
the mean curvature vector decays exponentially. The following eigenvalue
estimate is similar to (\cite{li}, Lemma 3.2). For the completeness, we sketch
the proof here.

\begin{lemma}
\label{L3.2} Let $(M,\lambda,g,\mathbf{T},J)$ be an $\eta$-Einstein Sasakian
manifold. Along the Lagrangian mean curvature flow, for any constants $T,$
$\epsilon,$ $\gamma,$ $\Lambda>0$, if the solution $L_{t}$ satisfies%
\begin{equation}%
\begin{array}
[c]{c}%
|A|\leq\Lambda,\text{\ }|H|+|\nabla H|\leq\epsilon e^{-\gamma t},\text{\ }%
t\in\lbrack0,T],
\end{array}
\label{82}%
\end{equation}
then the first eigenvalue $\lambda_{1}(t)$ satisfies%
\begin{equation}%
\begin{array}
[c]{c}%
\sqrt{\lambda_{1}(t)}\geq e^{-\frac{1}{2\gamma}(2\Lambda\epsilon+\epsilon
^{2})}\sqrt{\lambda_{1}(0)}-\frac{\Lambda\epsilon}{\gamma},\text{\ }%
t\in\lbrack0,T].
\end{array}
\label{83}%
\end{equation}

\end{lemma}

\begin{proof}
Let $\varphi(t,x)$ be a eigenfunction of the Laplacian operator $\Delta$ with
respect to the induced metric on $L_{t}$ satisfying%
\begin{equation}%
\begin{array}
[c]{c}%
-\Delta\varphi=\lambda_{1}(t)\varphi,\text{ }\int_{L_{t}}\varphi^{2}d\mu
_{t}=1.
\end{array}
\label{83a}%
\end{equation}
By integrating the first equation, the corresponding first eigenvalue
$\lambda_{1}(t)$ will satisfy%
\[%
\begin{array}
[c]{c}%
\lambda_{1}(t)=\int_{L_{t}}\left\vert \nabla\varphi\right\vert ^{2}d\mu_{t}.
\end{array}
\]
Taking derivative with respect to $t$ about the second equation of
(\ref{83a}), we have%
\begin{equation}%
\begin{array}
[c]{c}%
\int_{L_{t}}[2\varphi\frac{\partial}{\partial t}\varphi-\varphi^{2}%
|H|^{2}]d\mu_{t}=0.
\end{array}
\label{84}%
\end{equation}
Then, by using $\frac{\partial}{\partial t}\Delta=\frac{\partial}{\partial
t}g^{_{ij}}\nabla_{i}\nabla_{j}=2H_{k}h^{kij}\nabla_{i}\nabla_{j}$ and
(\ref{84}), one obtain
\begin{equation}%
\begin{array}
[c]{lll}%
\frac{\partial}{\partial t}\lambda_{1}(t) & = & -\frac{\partial}{\partial
t}\int_{L_{t}}\varphi\Delta\varphi d\mu_{t}\\
& = & -\int_{L_{t}}[2(\frac{\partial}{\partial t}\varphi)\Delta\varphi
+\varphi(\frac{\partial}{\partial t}\Delta)\varphi-\varphi\Delta\varphi
|H|^{2}]d\mu_{t}\\
& = & \int_{L_{t}}[\lambda_{1}(2\varphi\frac{\partial}{\partial t}%
\varphi-\varphi^{2}|H|^{2})-2H_{k}h^{kij}\varphi\nabla_{i}\nabla_{j}%
\varphi]d\mu_{t}\\
& = & \int_{L_{t}}-2H_{k}h^{kij}\varphi\nabla_{i}\nabla_{j}\varphi d\mu_{t}\\
& = & 2\int_{L_{t}}[H_{k}h^{kij}\nabla_{i}\varphi+\nabla_{i}(H_{k}%
h^{kij})\varphi]\nabla_{j}\varphi d\mu_{t}.
\end{array}
\label{85}%
\end{equation}
From the Codazzi equation we have%
\[%
\begin{array}
[c]{c}%
\nabla_{i}h_{j}^{ik}-\nabla_{j}h_{i}^{ik}=-\frac{1}{2}g^{kl}R_{\alpha\beta
}F_{l}^{\alpha}v_{j}^{\beta}=-\frac{1}{2}Kg^{kl}g_{\alpha\beta}F_{l}^{\alpha
}v_{j}^{\beta}=0,
\end{array}
\]
if $M$ is an $\eta$-Einstein Sasakian manifold. Combining this with (\ref{85})
and integration by parts%
\[%
\begin{array}
[c]{lll}%
\frac{\partial}{\partial t}\lambda_{1}(t) & = & 2\int_{L_{t}}[H_{k}%
h^{kij}\nabla_{i}\varphi+(H_{k}h^{kij})_{i}\varphi]\nabla_{j}\varphi d\mu
_{t}\\
& = & 2\int_{L_{t}}[H_{k}h^{kij}\nabla_{i}\varphi+(\nabla_{i}H_{k}%
h^{kij}+H_{k}\nabla_{j}H^{k})\varphi]\nabla_{j}\varphi d\mu_{t}\\
& = & 2\int_{L_{t}}[H_{k}h^{kij}\nabla_{i}\varphi+\nabla_{i}H_{k}%
h^{kij}\varphi]\nabla_{j}\varphi d\mu_{t}\\
&  & -\int_{L_{t}}|H|^{2}(\left\vert \nabla\varphi\right\vert ^{2}%
+\varphi\Delta\varphi)d\mu_{t}.
\end{array}
\]
Under the assumption (\ref{82}) and $\int_{L_{t}}\left\vert \varphi
\nabla\varphi\right\vert d\mu_{t}\leq(\int_{L_{t}}\varphi^{2}d\mu_{t}%
)^{\frac{1}{2}}(\int_{L_{t}}\left\vert \nabla\varphi\right\vert ^{2}d\mu
_{t})^{\frac{1}{2}}=\lambda_{1}^{\frac{1}{2}}(t)$, we have%
\[%
\begin{array}
[c]{l}%
\frac{\partial}{\partial t}\lambda_{1}(t)\geq-2\Lambda\epsilon e^{-\gamma
t}\lambda_{1}(t)-2\Lambda\epsilon e^{-\gamma t}\lambda_{1}^{\frac{1}{2}%
}(t)-2\epsilon^{2}e^{-2\gamma t}\lambda_{1}(t),
\end{array}
\]
that is
\[%
\begin{array}
[c]{l}%
\frac{\partial}{\partial t}\lambda_{1}^{\frac{1}{2}}(t)\geq-(\Lambda\epsilon
e^{-\gamma t}+\epsilon^{2}e^{-2\gamma t})\lambda_{1}^{\frac{1}{2}}%
(t)-\Lambda\epsilon e^{-\gamma t}%
\end{array}
\]
for $t\in\lbrack0,T].$ Thus, the first eigenvalue $\lambda_{1}(t)$ will
satisfy%
\[%
\begin{array}
[c]{c}%
e^{-\frac{1}{2\gamma}(2\Lambda\epsilon e^{-\gamma t}+\epsilon^{2}e^{-2\gamma
t})}\lambda_{1}^{\frac{1}{2}}(t)\geq e^{-\frac{1}{2\gamma}(2\Lambda
\epsilon+\epsilon^{2})}\lambda_{1}^{\frac{1}{2}}(0)+\frac{\Lambda\epsilon
}{\gamma}\int_{1}^{e^{-\gamma t}}e^{-\frac{1}{2\gamma}(2\Lambda\epsilon
u+\epsilon^{2}u^{2})}du,
\end{array}
\]
which implies%
\[%
\begin{array}
[c]{c}%
\lambda_{1}^{\frac{1}{2}}(t)\geq e^{-\frac{1}{2\gamma}(2\Lambda\epsilon
+\epsilon^{2})}\lambda_{1}^{\frac{1}{2}}(0)-\frac{\Lambda\epsilon}{\gamma}%
\end{array}
\]
for $t\in\lbrack0,T].$
\end{proof}

\begin{definition}
Let $(M,\lambda,g,\mathbf{T},J)$ be an $\eta$-Einstein Sasakian manifold. For
any positive constants $\delta,$ $\kappa,$ $r,$ $\Lambda,$ $\epsilon$, we
definition the following subspace of Legendrian submanifolds in $M^{2n+1}$ by
\[%
\begin{array}
[c]{ccc}%
\mathcal{B}(\kappa,r,\delta,\Lambda,\epsilon) & = & \left\{  L:L\text{ is
}\kappa\text{-noncollapsed on the scale }r\text{ with }\right. \\
&  & \left.  \text{ \ \ \ \ \ \ \ \ }\lambda_{1}\geq K+2+\delta,\text{
}|A|(t)\leq\Lambda,\text{ }|H|(t)\leq\epsilon\right\}  .
\end{array}
\]

\end{definition}

By using Lemma \ref{8} and Lemma \ref{9}, we can know that the following Lemma
is true.

\begin{lemma}
\label{L5.1}Let $(M^{2n+1},J,T)$ be an $\eta$-Einstein Sasakian manifold and
the solution $L_{t}$ $(t\in\lbrack0,T_{0}])$ of the Legendrian mean curvature
flow with exact mean curvature form $H_{0}$ for $t=0$. If the initial
Legendrian submanifold $L_{0}\in\mathcal{B}(\kappa,r,\delta,\Lambda,\epsilon
)$, then there exists $\tau=\tau(n,\Lambda,\delta,K_{2},K+2)>0$, such that
\[%
\begin{array}
[c]{c}%
L_{t}\in\mathcal{B}(\frac{1}{2}\kappa,r,\frac{2}{3}\delta,2\Lambda
,2\epsilon)\text{\ for }t\in\lbrack0,\tau].
\end{array}
\]

\end{lemma}

So we can obtain the following results.

\begin{lemma}
\label{16a}Let $(M,\lambda,g,\mathbf{T},J)$ be an $\eta$-Einstein Sasakian
manifold and the solution $L_{t}$ $(t\in\lbrack0,T_{0}])$ of the Legendrian
mean curvature flow with exact mean curvature form $H_{0}$ for $t=0$. For any
$\kappa_{0},$ $r_{0},$ $\Lambda_{0},$ $V_{0},$ $T_{0}>0$, assume that there
exists a constant $\epsilon_{0}=\epsilon_{0}(\kappa_{0},r_{0},\Lambda
_{0},V_{0})$ such that%
\[%
\begin{array}
[c]{l}%
(a).\text{ }L_{0}\in\mathcal{B}(\kappa_{0},r_{0},\delta_{0},\Lambda
_{0},\epsilon_{0})\text{\ and \textrm{Vol}}(L_{0})\leq V_{0},\\
(b).\text{ }L_{t}\in\mathcal{B}(\frac{1}{3}\kappa_{0},r_{0},\frac{1}{3}%
\delta_{0},6\Lambda_{0},2\epsilon_{0}^{\frac{1}{n+2}}),\text{ }(t\in
\lbrack0,T_{0}])
\end{array}
\]
Then we have\newline(1) The mean curvature vector satisfies
\[%
\begin{array}
[c]{c}%
\max_{L_{t}}|H|(t)\leq\epsilon_{0}^{\frac{1}{n+2}}e^{\frac{-\delta_{0}%
}{2(n+2)}t},\text{\ }t\in\lbrack\tau,T_{0}].
\end{array}
\]
(2) The second fundamental form
\[%
\begin{array}
[c]{c}%
\max_{L_{t}}|A|(t)\leq3\Lambda_{0},\text{ }t\in\lbrack0,T_{0}].
\end{array}
\]
(3) $L_{t}$ is $\frac{2}{3}\kappa_{0}$-noncollapsed on the scale $r_{0}$ for
$t\in\lbrack0,T_{0}]$.\newline(4) The first eigenvalue
\[%
\begin{array}
[c]{c}%
\lambda_{1}(t)\geq K+2+\frac{1}{2}\delta_{0},\text{ }t\in\lbrack0,T_{0}].
\end{array}
\]
Thus, the solution $L_{t}\in\mathcal{B}(\frac{2}{3}\kappa_{0},r_{0},\frac
{1}{2}\delta_{0},3\Lambda_{0},\epsilon_{0}^{\frac{1}{n+2}})$ for $t\in
\lbrack0,T_{0}]$, and by using Lemma \ref{L5.1}, we obtain that
\[%
\begin{array}
[c]{c}%
L_{t}\in\mathcal{B}(\frac{1}{3}\kappa_{0},r_{0},\frac{1}{3}\delta_{0}%
,6\Lambda_{0},2\epsilon_{0}^{\frac{1}{n+2}}),\text{ }t\in\lbrack0,T_{0}%
+\delta]
\end{array}
\]
for some\ $\delta=\delta(n,\Lambda_{0},K_{1})>0.$
\end{lemma}

\begin{proof}
The proof is similar to the proof of Lemma 5.2 in (\cite{li}), the reader can
refer to (\cite{li}).
\end{proof}

{\large Proof of Theorem \ref{Thm 3.2}:}

\begin{proof}
Suppose that $L_{0}\in\mathcal{B}(\kappa_{0},r_{0},\delta_{0},\Lambda
_{0},\epsilon_{0})$ for any positive constants $\kappa_{0},$ $r_{0},$
$\delta_{0},$ $\Lambda_{0}$ and $\epsilon_{0}$ small enough. Set%
\[%
\begin{array}
[c]{c}%
t_{0}=\sup\{t>0|\text{ }L_{s}\in\mathcal{B}(\frac{1}{3}\kappa_{0},r_{0}%
,\frac{1}{3}\delta_{0},6\Lambda_{0},2\epsilon_{0}^{\frac{1}{n+2}}),\text{
}s\in\lbrack0,t)\}.
\end{array}
\]
In the following, we will show that $t_{0}=\infty$. Suppose that $t_{0}%
<\infty$. By Lemma \ref{16a}, we know that%
\[%
\begin{array}
[c]{c}%
L_{t}\in\mathcal{B}(\frac{2}{3}\kappa_{0},r_{0},\frac{1}{2}\delta_{0}%
,3\Lambda_{0},\epsilon_{0}^{\frac{1}{n+2}}),
\end{array}
\]
for $t\in\lbrack0,t_{0})$ and $\epsilon_{0}=\epsilon_{0}(\kappa_{0}%
,r_{0},\delta_{0},\Lambda_{0},V_{0},K_{5},K+2)>0$. By using Lemma \ref{16a}
again, the solution $L_{t}$ can be extended to $[0,t_{0}+\delta]$ and $L_{t}$
is contained in $\mathcal{B}(\frac{1}{3}\kappa_{0},r_{0},\frac{1}{3}\delta
_{0},6\Lambda_{0},2\epsilon_{0}^{\frac{1}{n+2}})$ which contradicts the
definition of $t_{0}$. So we know that $t_{0}=\infty$ and%
\[%
\begin{array}
[c]{c}%
L_{t}\in\mathcal{B}(\frac{1}{3}\kappa_{0},r_{0},\frac{1}{3}\delta_{0}%
,6\Lambda_{0},2\epsilon_{0}^{\frac{1}{n+2}}),
\end{array}
\]
for $t\in\lbrack0,\infty).$ By using Lemma \ref{16a}, we obtain that the mean
curvature vector will decay exponentially to zero and the flow will converge
to a smooth minimal Legendrian submanifold. This completes this theorem.
\end{proof}

As a consequence of Theorem \ref{Thm 3.2} and Lemma \ref{L3.2}, we have the
Theorem \ref{Thm 3.1}.

\section{Deformation of minimal Legendrian submanifolds}

In this section, we study a Legendrian deformation of a minimal Legendrian
submanifold into an $\eta$-Einstein Sasakian manifold, and prove the
exponential decay of the mean curvature vector under the Legendrian mean
curvature flow with some special initial data.

In \cite{ono1}, H. Ono proved that a compact minimal Legendrian submanifold
$L$ into an $\eta$-Einstein Sasakian manifold with the $\eta$-Ricci constant
$K+2$ is Legendrian stable if and only if the first positive eigenvalue of the
Laplacian operator of $L$ satisfies $\lambda_{1}\geq K+2$. Thus, the
assumption (\ref{70}) of Theorem \ref{Thm 3.1} on the first eigenvalue of the
Laplacian operator ensures that the limit minimal Legendrian submanifold is
strictly Legendrian stable.

\begin{theorem}
\label{Thm 4.0}Let $(M,\lambda,g,\mathbf{T},J)$ be an $\eta$-Einstein Sasakian
manifold with the $\eta$-Einstein constant $K+2\geq0$. Suppose that
$\varphi:L\rightarrow M$ be a compact minimal Legendrian submanifold with the
first eigenvalue of the Laplacian operator $\lambda_{1}=K+2$ and $X $ is an
essential Legendrian variation of $L_{0}=\varphi(L)$. Let $\varphi
_{s}:L\rightarrow M,$ $s\in(-\delta,\delta),$ with $\varphi_{0}=\varphi$ be a
one-parameter family of Legendrian deformations generated by $X$. Then there
exists $\epsilon_{0}=\epsilon_{0}(X,L_{0},M)>0 $ such that if $L_{s}%
=\varphi_{s}(L)\subset M$ satisfying%
\[%
\begin{array}
[c]{c}%
\left\Vert \varphi_{s}-\varphi_{0}\right\Vert _{C^{3}}\leq\epsilon_{0}%
\end{array}
\]
then the Legendrian mean curvature flow with the initial Legendrian
submanifold $L_{s}$ will converge exponentially fast to a minimal Legendrian
submanifold in $M$.
\end{theorem}

Note that we can show that a minimal Legendrian submanifold $L$ with
$\lambda_{1}=K+2$ is strictly Legendrian stable along an essential Legendrian
variation $X$ as in Lemma \ref{L40}. Theorem \ref{Thm 4.0} says that the flow
will exist for all time and converge if the initial Legendrian submanifold is
a small perturbation of $L$ along essential Legendrian variations.

\begin{definition}
\label{d2} Let $L\subset M$ be a Legendrian submanifold and $X$ be a vector
field along $L$. $X$ is called a Legendrian variation vector field if
\[%
\begin{array}
[c]{c}%
X=J\nabla f+2f\mathbf{T},
\end{array}
\]
for some smooth function $f$ on $L$. A smooth family $\varphi_{s}$ of
immersions of $L$ into $M$ is called a Legendrian deformation if its
derivative $X=\frac{\partial\varphi_{s}}{\partial s}(L)$ is Legendrian
variation for each $s$.
\end{definition}

In the following, we assume that $\varphi_{0}:L\rightarrow M$ is a smooth
minimal Legendrian submanifold into an $\eta$-Einstein Sasakian manifold
$(M^{2n+1},J,\mathbf{T})$ with the $\eta$-Ricci constant $K+2$, and $X=J\nabla
f_{0}+2f_{0}\mathbf{T}$ is a Legendrian variation of $L_{0}=\varphi_{0}(L)$
for some smooth function $f_{0}$ on $L$. We extend the vector $X$ into a
neighborhood of $L_{0}$ in $M$ such that it is still Legendrian variation. Let
$\varphi_{s}:L\rightarrow M,$ $s\in(-\delta,\delta),$ be a family of
Legendrian deformations generated by $X$ and we write $L_{s}=\varphi_{s}%
(L_{0})$. For such a Legendrian deformation $L_{s}$, we can find smooth
functions $f_{s}$ such that%
\begin{equation}%
\begin{array}
[c]{c}%
\frac{\partial}{\partial s}L_{s}=J\nabla f_{s}+2f_{s}\mathbf{T},
\end{array}
\label{401}%
\end{equation}
then the Legendrian angle $\alpha_{s}$ of $L_{s}$ satisfies%
\begin{equation}%
\begin{array}
[c]{c}%
\frac{\partial}{\partial s}\alpha_{s}=\Delta_{s}\alpha_{s}+(K+2)\alpha_{s}.
\end{array}
\label{402}%
\end{equation}

Recall that a minimal Legendrian submanifold is called Legendrian stable
(resp. strictly stable), if for any Legendrian variation $X$ the second
variation along $X$ of the volume functional is nonnegative (resp. positive).
When $M$ is a $\eta$-Einstein Sasakian manifold with the $\eta$-Ricci constant
$K+2$, Ono \cite{ono1} proved that a compact minimal Legendrian submanifold
$L$ is Legendrian stable if and only if the first eigenvalue of the Laplacian
operator of $L$ satisfies $\lambda_{1}\geq K+2.$

\begin{definition}
A nonzero vector field $X$ is called an essential Legendrian variation of a
Legendrian submanifold $L_{0}$ if $X$ can be written as $X=J\nabla
f+2f\mathbf{T}$ for $f$ is not a first eigenfunction of the Laplacian operator
$\Delta$ on $L$.
\end{definition}

For an essential Legendrian vector $X$ on a minimal Legendrian submanifold
$L_{0}$, we can show that $L_{0}$ is strictly Legendrian stable along the
Legendrian variation $X$ in the following sense:

\begin{lemma}
\label{L40}Let $\varphi_{s}:L\rightarrow M$ be a Legendrian deformation of a
minimal Legendrian submanifold $L_{0}=\varphi_{0}(L)$ with the first
eigenvalue of the Laplacian operator $\Delta_{0}$ on $L_{0}$ has $\lambda
_{1}=K+2$ and $X=\frac{\partial\varphi}{\partial s}|_{s=0}$. Then $X$ is an
essential Legendrian variation on $L_{0}$ if and only if%
\[%
\begin{array}
[c]{c}%
\frac{d^{2}}{ds^{2}}\mathrm{Vol}(L_{s})|_{s=0}>0.
\end{array}
\]

\end{lemma}

\begin{proof}
Let $X_{s}=\frac{\partial\varphi_{s}}{\partial s}$ and since $L_{s}$ is a
Legendrian deformation, we can find smooth functions $f_{s}$ on $L_{s}$ such
that $X_{s}=J\nabla f_{s}+2f_{s}\mathbf{T}.$ From (\ref{2}) we have
$\frac{\partial}{\partial s}g_{ij}=2\nabla^{k}f_{s}h_{kij}$ and $H_{k}%
=\nabla_{k}\alpha_{s},$ we compute%
\[%
\begin{array}
[c]{c}%
\frac{d}{ds}\mathrm{Vol}(L_{s})=\int_{L_{s}}\frac{1}{2}g^{ij}\frac{\partial
g_{ij}}{\partial s}d\mu_{s}=\int_{L_{s}}\nabla^{k}f_{s}\nabla_{k}\alpha
_{s}d\mu_{s}=-\int_{L_{s}}\alpha_{s}\Delta_{s}f_{s}d\mu_{s}.
\end{array}
\]
Then the second variation of the volume is%
\begin{equation}%
\begin{array}
[c]{c}%
\frac{d^{2}}{ds^{2}}\mathrm{Vol}(L_{s})|_{s=0}=\int_{L_{0}}(\Delta_{0}%
f_{0}+(K+2)f_{0})\Delta_{0}f_{0}d\mu_{0}%
\end{array}
\label{403}%
\end{equation}
where we used the equality (\ref{402}) and the fact that $L_{0}$ is minimal
Legendrian. By the eigenvalue decomposition, we can assume that%
\[%
\begin{array}
[c]{c}%
f_{0}=\sum_{i=1}^{\infty}a_{i}\phi_{i},
\end{array}
\]
where the eigenfunctions $\phi_{i}$ satisfies%
\[%
\begin{array}
[c]{c}%
-\Delta_{0}\phi_{i}=\lambda_{i}\phi_{i},\text{ }\int_{L_{0}}\phi_{i}^{2}%
d\mu_{0}=1
\end{array}
\]
for the corresponding eigenvalues $0<\lambda_{1}<\lambda_{2}<\cdots$. Thus,
(\ref{403}) can be written as%
\[%
\begin{array}
[c]{c}%
\frac{d^{2}}{ds^{2}}\mathrm{Vol}(L_{s})|_{s=0}=\sum_{i=1}^{\infty}a_{i}%
^{2}\lambda_{i}(\lambda_{i}-(K+2))\geq0,
\end{array}
\]
and the equality holds if and only if $a_{i}=0$ for all $i\geq2$ and thus
$f_{0}$ is the first eigenfunction of the Laplacian operator $\Delta_{0}$
corresponding to the first eigenvalue $\lambda_{1}.$
\end{proof}

We denote by $L_{s,t}=\varphi_{s,t}(L_{0})$, $t\in\lbrack0,T_{0}],$ the
Legendrian mean curvature flow with the initial data $L_{s}$. Since $L_{s}$ is
a Legendrian deformation of $L_{0}$, the mean curvature form of $L_{s}$ is
exact for each $s$. Thus, the mean curvature form of $L_{s,t}$ is also exact,
and denote the Legendrian angle by $\alpha_{s,t}$. Suppose that the
deformation $L_{s}$ is sufficiently close to $L_{0}$ in the following sense%
\begin{equation}%
\begin{array}
[c]{c}%
\left\Vert \varphi_{s}-\varphi_{0}\right\Vert _{C^{3}}\leq\epsilon_{0}%
\end{array}
\label{404}%
\end{equation}
for some sufficient small $\epsilon_{0}$. The following Lemma show that
$\alpha_{s,t}$ satisfies certain inequality if $L_{s}$ is sufficiently close
to $L_{0}$.

\begin{lemma}
\label{lemma 4.0}Let $X=J\nabla f_{0}+2f_{0}\mathbf{T}$ be an essential
Legendrian variation of $L_{0}$, where $L_{0}$ is a minimal Legendrian
submanifold with the first eigenvalue $\lambda_{1}=K+2$. For any $\Lambda>0 $,
there exists $\epsilon_{0}=\epsilon_{0}(L_{0},X,M)>0$ and $\delta_{0}>0 $ such
that if $L_{s,t}$ satisfies%
\begin{equation}%
\begin{array}
[c]{c}%
|A_{s}|(t)\leq\Lambda,\text{\ }|H_{s}|(t)\leq\epsilon_{0},\text{ }%
\mathrm{for}\text{ }t\in\lbrack0,T_{0}]
\end{array}
\label{405}%
\end{equation}
then the Legendrian angle $\alpha_{s,t}$ of $L_{s,t}$ satisfies
\begin{equation}%
\begin{array}
[c]{c}%
\int_{L_{s,t}}\left\vert \Delta\alpha_{s,t}\right\vert ^{2}d\mu_{_{s,t}}%
\geq(K+2+\delta_{0})\int_{L_{s,t}}\left\vert \nabla\alpha_{s,t}\right\vert
^{2}d\mu_{_{s,t}},\text{ }t\in\lbrack0,T_{0}].
\end{array}
\label{406}%
\end{equation}
Therefore, we have%
\begin{equation}%
\begin{array}
[c]{c}%
\frac{\partial}{\partial t}\int_{L_{s,t}}\left\vert H_{s,t}\right\vert
^{2}d\mu_{_{s,t}}\leq-2(\delta_{0}-\Lambda\epsilon_{0})\int_{L_{s,t}%
}\left\vert H_{s,t}\right\vert ^{2}d\mu_{_{s,t}},\text{\ }t\in\lbrack0,T_{0}].
\end{array}
\label{407}%
\end{equation}

\end{lemma}

\begin{proof}
Suppose that (\ref{406}) doesn't hold, there exist some constants
$s_{i}\rightarrow0$, $\delta_{i}\rightarrow0$ and $t_{i}\in\lbrack0,T_{0}]$
such that
\begin{equation}%
\begin{array}
[c]{c}%
|A_{s_{i}}|(t)\leq\Lambda,\text{\ }|H_{s_{i}}|(t)\leq\epsilon_{0},\text{
}\mathrm{for}\text{ }t\in\lbrack0,T_{0}]
\end{array}
\label{408}%
\end{equation}
and%
\begin{equation}%
\begin{array}
[c]{c}%
\int_{L_{s_{i},t_{i}}}\left\vert \Delta\alpha_{s_{i},t_{i}}\right\vert
^{2}d\mu_{s_{i},t_{i}}\leq(K+2+\delta_{i})\int_{L_{s_{i},t_{i}}}\left\vert
\nabla\alpha_{s_{i},t_{i}}\right\vert ^{2}d\mu_{_{s_{i},t_{i}}}.
\end{array}
\label{409}%
\end{equation}
By (\ref{408}), a sequence of the Legendrian mean curvature flow $L_{s_{i},t}
$, $t\in(0,T_{0})$, will converge to a limit Legendrian mean curvature flow
$L_{\infty,t}$ smoothly for $t\in(0,T_{0})$. Since the initial submanifolds
$L_{s_{i}}$ satisfies (\ref{404}), the limit flow has $L_{\infty
}(t)\rightarrow L_{0}$ in $C^{2,\alpha}$ as $t$ goes to zero. By (\ref{408})
again the mean curvature of $L_{\infty}(t)$, $t\in(0,T_{0})$, are identically
zero and by the uniqueness of mean curvature flow we have%
\[
L_{s_{i},t}\rightarrow L_{\infty,t}=L_{0},\text{\ }t\in\lbrack0,T_{0}].
\]

From (\ref{409}), we can take $s_{i}\rightarrow0$ and $\delta_{i}\rightarrow0$
to yield%
\[%
\begin{array}
[c]{c}%
\int_{L_{s_{i},t_{i}}}|\Delta\frac{1}{s_{i}}\alpha_{s_{i},t_{i}}|^{2}%
d\mu_{s_{i},t_{i}}\leq(K+2+\delta_{i})\int_{L_{s_{i},t_{i}}}|\nabla\frac
{1}{s_{i}}\alpha_{s_{i},t_{i}}|^{2}d\mu_{_{s_{i},t_{i}}},
\end{array}
\]
since $L_{0}$ is minimal, we have
\begin{equation}%
\begin{array}
[c]{c}%
\int_{L_{0}}|\Delta\frac{\partial\alpha_{s,t}}{\partial s}|_{(0,t_{i})}%
|^{2}d\mu_{0}\leq(K+2)\int_{L_{0}}|\nabla\frac{\partial\alpha_{s,t}}{\partial
s}|_{(0,t_{i})}|^{2}d\mu_{0}.
\end{array}
\label{410}%
\end{equation}

On the other hand, from (\ref{angle}) the Legendrian angle $\alpha_{s,t}$
satisfies
\[%
\begin{array}
[c]{c}%
\frac{\partial}{\partial t}\alpha_{s,t}=\Delta_{s,t}\alpha_{s,t}%
+(K+2)\alpha_{s,t}.
\end{array}
\]
Since $L_{0}$ is minimal, we can take $s_{i}\rightarrow0$ to get
\begin{equation}%
\begin{array}
[c]{c}%
\frac{\partial}{\partial t}\frac{\partial\alpha_{s,t}}{\partial s}%
|_{s=0}=\Delta_{0}\frac{\partial\alpha_{s,t}}{\partial s}|_{s=0}%
+(K+2)\frac{\partial\alpha_{s,t}}{\partial s}|_{s=0}.
\end{array}
\label{411}%
\end{equation}
By the eigenvalue decomposition as in Lemma \ref{L40}, we have
\[%
\begin{array}
[c]{c}%
\frac{\partial\alpha_{s,t}}{\partial s}|_{(s,t)=(0,0)}=-\Delta_{0}%
f_{0}-(K+2)f_{0}\text{ }\bot\text{ }E_{\lambda_{1}},
\end{array}
\]
where $E_{\lambda_{1}}$is the first eigenspace of the Laplacian operator
$\Delta_{0}$ on $L_{0}.$ For any function $\phi\in E_{\lambda_{1}}$ by
(\ref{411}) we have%
\[%
\begin{array}
[c]{ccc}%
\frac{\partial}{\partial t}\int_{L_{0}}\phi\frac{\partial\alpha_{s,t}%
}{\partial s}|_{s=0}d\mu_{0} & = & \int_{L_{0}}\phi\lbrack\Delta_{0}%
\frac{\partial\alpha_{s,t}}{\partial s}|_{s=0}+(K+2)\frac{\partial\alpha
_{s,t}}{\partial s}|_{s=0}]d\mu_{0}\\
& = & \int_{L_{0}}\left(  \Delta_{0}\phi+(K+2)\phi\right)  \frac
{\partial\alpha_{s,t}}{\partial s}|_{s=0}d\mu_{0}=0,
\end{array}
\]
which show that
\[%
\begin{array}
[c]{c}%
\frac{\partial}{\partial s}\alpha_{s,t}|_{s=0}\text{ }\bot\text{ }%
E_{\lambda_{1}},\text{\ \textrm{for}\ }t\in\lbrack0,T_{0}].
\end{array}
\]

Since $X$ is an essential Legendrian variation, we get
\[%
\begin{array}
[c]{c}%
\frac{\partial\alpha_{s,t}}{\partial s}|_{(s,t)=(0,0)}=-\Delta_{0}%
f_{0}-(K+2)f_{0}%
\end{array}
\]
is nonzero. Thus, by (\ref{411}) we can see that $\frac{\partial\alpha_{s,t}%
}{\partial s}|_{s=0}$ is nonzero and orthogonal to $E_{\lambda_{1}}$ for all
$t\in\lbrack0,T_{0}]$, which implies
\[%
\begin{array}
[c]{c}%
\int_{L_{0}}|\Delta\frac{\partial\alpha_{s,t}}{\partial s}|_{s=0}|^{2}d\mu
_{0}\geq\lambda_{2}\int_{L_{0}}|\nabla\frac{\partial\alpha_{s,t}}{\partial
s}|_{s=0}|^{2}d\mu_{0},
\end{array}
\]
where the second eigenvalue $\lambda_{2}$ is big than $\lambda_{1}=K+2$. This
is contradicts (\ref{410}), and (\ref{406}) is proved.

Now recall that by (\ref{44}) in the proof of Lemma \ref{11} with
$H_{s,t}=\nabla\alpha_{s,t},$ we have
\[%
\begin{array}
[c]{lll}%
\frac{1}{2}\frac{\partial}{\partial t}\int_{L_{s,t}}\left\vert H_{s,t}%
\right\vert ^{2}d\mu_{_{s,t}} & \leq & \int_{L_{s,t}}[-(\Delta\alpha
_{s,t})^{2}+(K+2+\Lambda\epsilon_{0})\left\vert H_{s,t}\right\vert ^{2}%
]d\mu_{_{s,t}}\\
& = & \int_{L_{s,t}}[-(\Delta\alpha_{s,t})^{2}+(K+2+\Lambda\epsilon
_{0})\left\vert \nabla\alpha_{s,t}\right\vert ^{2}]d\mu_{_{s,t}}\\
& \leq & -(\delta_{0}-\Lambda\epsilon_{0})\int_{L_{s,t}}\left\vert
H_{s,t}\right\vert ^{2}d\mu_{_{s,t}},
\end{array}
\]
which is the inequality (\ref{407}).
\end{proof}

In the following we prove Theorem \ref{Thm 4.0} by using the same argument as
in the proof of Theorem \ref{23}. Since $L_{0}$ is a smooth minimal Legendrian
submanifold, we can find $\kappa_{0},$ $r_{0}>0$ such that $L_{0} $ is
$2\kappa_{0}$-noncollapsed on the scale $r_{0}$. Thus we can choose
$\epsilon_{0}$ small enough such that $L_{s}$ is $\kappa_{0}$-noncollapsed on
the scale $r_{0}$ and $L_{s}\in\mathcal{A}(\kappa_{0},r_{0},\Lambda
_{0},\epsilon_{0})$ for some constant $\Lambda_{0}>0$, where $\mathcal{A}%
(\kappa,r,\Lambda,\epsilon)$ is the subspace of Legendrian submanifolds in $M$
defined as (\ref{LS}).

Consider the solution $L_{s,t}$ of the Legendrian mean curvature flow with the
initial data $L_{s}$. Using the Lemma \ref{8} and Lemma \ref{9}, we can get
the following result.

\begin{lemma}
\label{lemma 4.1} If the initial Legendrian submanifold $L_{s}\in
\mathcal{A}(\kappa,r,\Lambda,\epsilon)$, then there exists $\tau
=\tau(n,\Lambda,K_{1})$ such that $L_{s,t}\in\mathcal{A}(\frac{1}{2}%
\kappa,r,2\Lambda,2\epsilon)$ for $t\in\lbrack0,\tau]$.
\end{lemma}

Then by applying the same argument as in the proof of Lemma \ref{16} and using
the results in Lemma \ref{lemma 4.0}, we can obtain the following crucial Lemma.

\begin{lemma}
\label{lemma 4.2} For any $\kappa_{0},$ $r_{0},$ $\Lambda_{0},$ $V_{0},$
$T_{0}>0$, there exists $\epsilon_{0}=\epsilon_{0}(\kappa_{0},r_{0}%
,\Lambda_{0},V_{0})$ and $\delta_{0}>0$ such that if the solution $L_{s,t},$
$t\in\lbrack0,T_{0}],$ of the Legendrian mean curvature flow satisfies
\[%
\begin{array}
[c]{l}%
1.\text{ }L_{s}\in\mathcal{A}(\kappa_{0},r_{0},\Lambda_{0},\epsilon
_{0})\text{\ and \textrm{Vol}}(L_{s})\leq V_{0},\\
2.\text{ }L_{s,t}\in\mathcal{A}(\frac{1}{3}\kappa_{0},r_{0},6\Lambda
_{0},2\epsilon_{0}^{\frac{1}{n+2}}),\text{ }t\in\lbrack0,T_{0}].
\end{array}
\]
Then we have\newline(1) The mean curvature vector satisfies
\[%
\begin{array}
[c]{c}%
\max_{L_{s,t}}|H_{s,t}|\leq\epsilon_{0}^{\frac{1}{n+2}}e^{-\frac{\delta_{0}%
}{n+2}t},\text{\ }t\in\lbrack\tau,T_{0}].
\end{array}
\]
(2) The second fundamental form
\[%
\begin{array}
[c]{c}%
\max_{L_{s,t}}|A_{s,t}|\leq3\Lambda_{0},\text{ }t\in\lbrack0,T_{0}].
\end{array}
\]
(3) $L_{s,t}$ is $\frac{2}{3}\kappa_{0}$-noncollapsed on the scale $r_{0}$ for
$t\in\lbrack0,T_{0}]$.\newline Thus, the solution $L_{s,t}\in\mathcal{A}%
(\frac{2}{3}\kappa_{0},r_{0},3\Lambda_{0},\epsilon_{0}^{\frac{1}{n+2}})$ for
$t\in\lbrack0,T_{0}]$, and by using Lemma \ref{lemma 4.1}, we obtain that
\[%
\begin{array}
[c]{c}%
L_{s,t}\in\mathcal{A}(\frac{1}{3}\kappa_{0},r_{0},6\Lambda_{0},2\epsilon
_{0}^{\frac{1}{n+2}}),\text{ }t\in\lbrack0,T_{0}+\delta]
\end{array}
\]
for some\ $\delta=\delta(n,\Lambda_{0},K_{1})>0.$
\end{lemma}

{\large Proof of Theorem \ref{Thm 4.0}:}

\begin{proof}
Suppose that $L_{s}\in\mathcal{A}(\kappa_{0},r_{0},\Lambda_{0},\epsilon_{0})$
for any positive constants $\kappa_{0},$ $r_{0},$ $\Lambda_{0}$ and
$\epsilon_{0}$ small enough. Set%
\[%
\begin{array}
[c]{c}%
t_{0}=\sup\{t>0|\text{ }L_{s,l}\in\mathcal{A}(\frac{1}{3}\kappa_{0}%
,r_{0},6\Lambda_{0},2\epsilon_{0}^{\frac{1}{n+2}}),\text{ }l\in\lbrack0,t)\}.
\end{array}
\]
In the following, we will show that $t_{0}=\infty$. Suppose that $t_{0}%
<\infty$. By Lemma \ref{lemma 4.2}, we know that%
\[%
\begin{array}
[c]{c}%
L_{s,t}\in\mathcal{A}(\frac{2}{3}\kappa_{0},r_{0},3\Lambda_{0},\epsilon
_{0}^{\frac{1}{n+2}}),
\end{array}
\]
for $t\in\lbrack0,t_{0})$ and $\epsilon_{0}=\epsilon_{0}(\kappa_{0}%
,r_{0},\Lambda_{0},V_{0},K_{5})>0$. By using Lemma \ref{lemma 4.2} again, the
solution $L_{s,t}$ can be extended to $[0,t_{0}+\delta]$ and $L_{s,t}$ is
contained in $\mathcal{A}(\frac{1}{3}\kappa_{0},r_{0},6\Lambda_{0}%
,2\epsilon_{0}^{\frac{1}{n+2}})$ which contradicts the definition of $t_{0}$.
So we know that $t_{0}=\infty$ and%
\[%
\begin{array}
[c]{c}%
L_{s,t}\in\mathcal{A}(\frac{1}{3}\kappa_{0},r_{0},6\Lambda_{0},2\epsilon
_{0}^{\frac{1}{n+2}}),
\end{array}
\]
for $t\in\lbrack0,\infty).$ By using Lemma \ref{lemma 4.2}, we obtain that the
mean curvature vector will decay exponentially to zero and the flow will
converge to a smooth minimal Legendrian submanifold. This completes the theorem.
\end{proof}

\end{document}